\documentclass[11pt,letter]{article}
\usepackage[latin1]{inputenc}
\usepackage{fancyhdr}
\usepackage{verbatim}
\usepackage{indentfirst}
\usepackage{graphicx}
\usepackage{epstopdf}
\usepackage{color}
\usepackage{newlfont}
\usepackage{amssymb}
\usepackage[intlimits]{amsmath}
\usepackage{latexsym}
\usepackage{amsthm}
\usepackage{amscd}
\usepackage[dvips, outer=1in, top=1in, bottom=1in]{geometry}
\usepackage{multirow}
\usepackage{hyperref}

\theoremstyle{plain}                    
\newtheorem{theorem}{Theorem}[section]     
\newtheorem{lem}[theorem]{Lemma}            
\theoremstyle{definition}

\newtheorem{remark}[theorem]{Remark}
\newtheorem{prop}[theorem]{Proposition}
\newtheorem{cor}[theorem]{Corollary}

\newcommand{\de}{\mathrm{d}}
\newcommand{\virg}[1]{``#1''}
\newcommand{\N}{\mathbb{N}}
\newcommand{\Z}{\mathbb{Z}}
\newcommand{\Q}{\mathbb{Q}}
\newcommand{\R}{\mathbb{R}}
\newcommand{\C}{\mathbb{C}}

\newcommand{\bianco}{\textcolor{white}{.}}

\newcommand{\les}{\left[\left[}
\newcommand{\res}{\right]\right]}
\newcommand{\X}{(0,1]\smallsetminus\mathbb{Q}}
\newcommand{\su}[2]{\mathcal{S}_{#1}(#2)}
\newcommand{\s}[3]{\mathcal{S}_{#1}^{(#2)}(#3)}

\newcommand{\ha}{\frac{1}{2}}
\newcommand{\cu}{{(\mbox{\tiny{curl}})}}
\newcommand{\fl}{{(\mbox{\tiny{flat}})}}
\newcommand{\pifi}{\frac{\pi}{4}i}

\input xy
\xyoption{all}
\begin{document}
\clearpage{\pagestyle{empty}\cleardoublepage}
\title{Limiting Curlicue Measures for Theta Sums}
\author{Francesco Cellarosi\footnote{Mathematics Department, Princeton University, Princeton, NJ, U.S.A. \texttt{fcellaro@math.princeton.edu}}}
\maketitle
\begin{abstract}
We consider the ensemble of curves $\{\gamma_{\alpha,N}:\:\alpha\in(0,1],\:N\in\N\}$ obtained by linearly interpolating the values of the 
normalized theta sum $N^{-\ha}\sum_{n=0}^{N'-1}\exp(\pi i n^2\alpha)$, $0\leq N'<N$. We prove the existence of limiting finite-dimensional distributions for such curves as $N\rightarrow\infty$, with respect to an absolutely continuous probability measure $\mu_R$ 
on $(0,1]$. Our Main Theorem generalizes a result by Marklof \cite{Marklof1999b} and Jurkat and van Horne \cite{Jurkat-vanHorne1981-proofCLT,Jurkat-vanHorne1982-CLT}.
Our proof relies on the analysis of the geometric structure of such curves, which exhibit spiral-like patterns (\emph{curlicues}) at different scales. We exploit a renormalization procedure constructed by means of the continued fraction expansion of $\alpha$ with even partial quotients and a renewal-type limit theorem for the denominators of such continued fraction expansions.
\end{abstract}
\section{Introduction}
Given $a\in(-1,1]\smallsetminus\{0\}$ and $N\in\N$ consider the theta 
sum
\begin{equation}\label{sum}
\su{a}{N}:=\sum_{n=0}^{N-1}\exp(\pi in^2a)\in\C.
\end{equation}
For arbitrary $L\geq0$ let us define it as
\begin{equation}\nonumber
\su{a}{L}:=\sum_{n=0}^{\lfloor L\rfloor-1}\exp(\pi i n^2a)+\{L\}\exp(\pi i\lfloor L\rfloor^2a)\in\C,
\end{equation}
where $\lfloor\cdot\rfloor$ denotes the floor function and $\{\cdot\}$ 
the fractional part. One has $\su{a+2}{N}=\su{a}{N}$, $\su{-a}{N}=\overline{\su{a}{N}}$ and $\int_{-1}^1\left|\su{a}{N}\right|^2\de\,a=N$. It is convenient to consider $\alpha=|a|\in(0,1]$ and to study $\su{\alpha}{L}$, 
see Section \ref{AERFs}.

Our goal is to study the curves generated by theta 
 sums: i.e. $$\gamma=\gamma_{\alpha,N}:[0,1]\rightarrow\C\simeq\R^2,\hspace{.5cm}t\mapsto\frac{\su{\alpha}{tN}}{\sqrt{N}}$$ as $N\rightarrow\infty$. Such curves are piecewise linear, of length $\sqrt{N}$ (being made of $N$ segments of length $N^{-\ha}$). In particular we are interested in the ensemble of curves $\{\gamma_{\alpha,N}\}_{\alpha\in(0,1]}$ as $N\rightarrow\infty$ when $\alpha$ is distributed according to some probability measure on $(0,1]$.

\begin{figure}[h!]
\begin{center}
\vspace{-1.cm}
\includegraphics[width=12.5cm, angle=0]{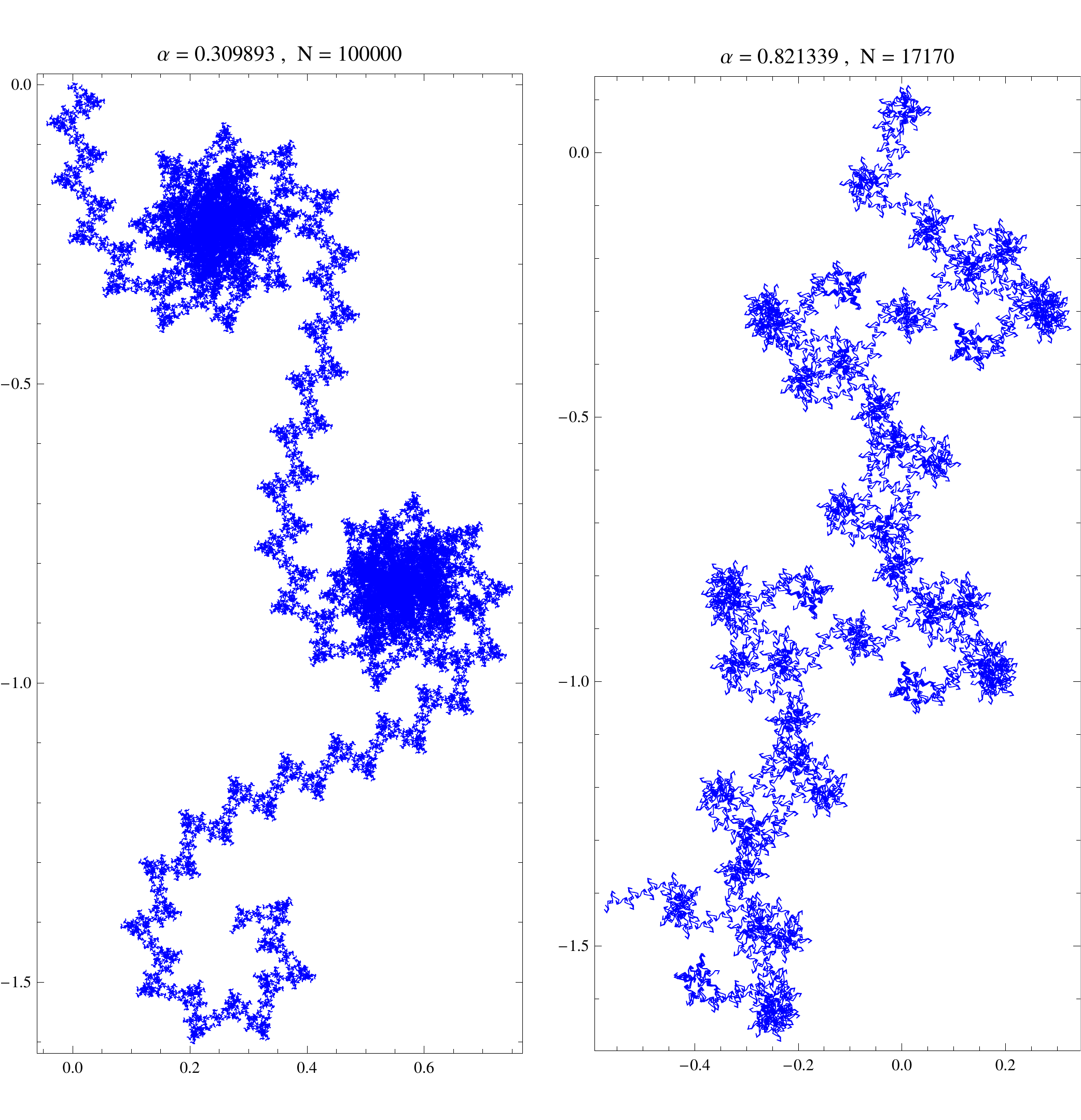}
\hspace{.8cm}
\includegraphics[width=12cm, angle=0]{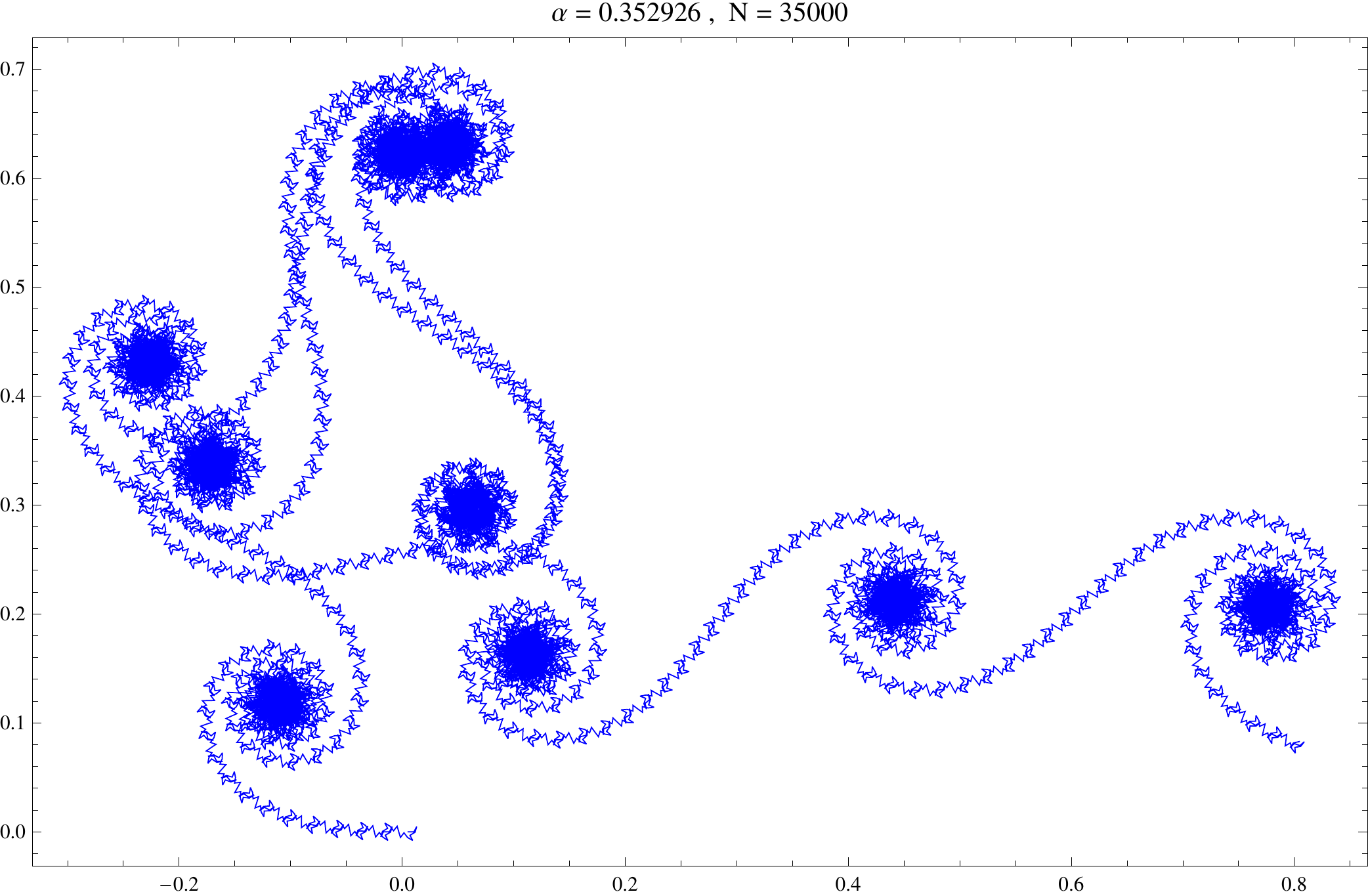}
\vspace{-.3cm}
\caption{\small{Three curves of the form $t\mapsto\gamma_{\alpha,N}(t)$.}}\label{fig1}
\end{center}
\end{figure}
As illustrated in Figure \ref{fig1}, these curves exhibit a geometric multi-scale structure, including spiral-like fragments (\emph{curlicues}). For a discussion on the geometry of $t\mapsto\su{\alpha}{tN}$ (and more general curves defined using exponential sums) in connection with uniform distribution modulo 1, see Dekking and Mend\`{e}s France \cite{Dekking-MendesFrance}. For the study of other geometric and thermodynamical properties of such curves, see Mend\`{e}s France \cite{MendesFrance-entropy1981,MendesFrance-entropie1981-2} and Moore and van der Poorten \cite{Moore-vanderPoorten}.

Denote by $\mathcal B^k$ the Borel $\sigma$-algebra on $\C^k$ and let $\mu_R$ be the probability measure 
on $(0,1]$ whose density is $\frac{1}{\log 3}\left(\frac{1}{3-x}+\frac{1}{1+x}\right)$. In Section \ref{section: jump transformation} we shall see that the measure $\mu_R$ naturally appears in 
our situation.
\begin{theorem}[Main Theorem]\label{main-theorem} For every $k\in\N$, for every $t_1,\ldots,t_k\in[0,1]$, $0\leq t_1<t_2<\ldots<t_k\leq 1$, there exists a probability measure $
\mathrm P^{(k)}_{t_1,\ldots,t_k}$ on $\C^k$ such that for every open, nice $A\in\mathcal B^k$,
\begin{eqnarray}
\lim_{N\rightarrow\infty}\mu_R\left(\left\{\alpha\in(0,1]:\:\left(\gamma_{\alpha,N}(t_j)\right)_{j=1}^k\in A\right\}\right)=\mathrm P^{(k)}_{t_1,\ldots,t_k}(A).\label{eq: main-thm-A}
\end{eqnarray}
The measure $\mathrm P^{(k)}_{t_1,\ldots,t_k}$ is called \emph{curlicue measure} associated with the moments of time $t_1,\ldots,t_k$.
\end{theorem}
We shall define later what we mean by \virg{nice} and prove that many interesting sets are indeed nice. For example, if $B_{z}(\rho):=\{w\in\C:\:|z-w|<\rho\}$, then for every $(z_1,\ldots,z_k)\in\C^k$, the set $A=B_{z_1}(\rho_1)\times\ldots\times B_{z_k}(\rho_k)\subseteq\C^k$ is nice for all $(\rho_1,\ldots,\rho_k)\in\R_{>0}^k$, except possibly for a countable set.

Our main theorem generalizes a result by Marklof \cite{Marklof1999b} (corresponding to $k=1$ and $t_1=1$), which in particular implies the following theorem by Jurkat and van Horne \cite{Jurkat-vanHorne1981-proofCLT,Jurkat-vanHorne1982-CLT}. 
\begin{theorem}[Jurkat and van Horne]
There exists a function $\Psi(a,b)$ such that for all (except for countably many) $a,b\in\mathbb{R}$,
$$\lim_{N\rightarrow\infty}\left|\left\{\alpha:\:a< N^{-\ha}|\su{\alpha}{N}|< b\right\}\right|=\Psi(a,b).$$
\end{theorem}
Let us remark that Marklof's approach uses the equidistribution of long, closed horocycles in the unit tangent bundle of a suitably constructed non-compact hyperbolic manifold of finite volume. Moreover, 
the explicit asymptotics for the moments of $N^{-\ha}|\su{a}{N}|$ (along with central limit theorems \cite{Jurkat-vanHorne1981-proofCLT,Jurkat-vanHorne1982-CLT,Jurkat-VanHorne1983}) were 
found by Jurkat and van Horne and generalized by Marklof \cite{Marklof1999b} in the case of more general theta sums using Eisenstein series. In particular it is known that the above distribution function $\Psi$ is not Gaussian.
Thus far, our approach only shows existence of the limiting measures $\mathrm P_{t_1,\ldots,t_k}^{(k)}$. It is in principle possible to derive 
quantitative informations on the decay of their moments from our method too, but we shall not dwell on this. For a preliminary discussion of the present work, see Sinai \cite{Sinai-curlicues}.

\begin{remark}\label{first remark}
Consider the probability space $\left((0,1],\mathcal B,\mu_R\right)$, where $\mathcal B$ is the Borel $\sigma$-algebra on $(0,1]$ and $\mu_R$ is as above. 
We look at $\gamma_{\alpha,N}$ as a random function, i.e as a measurable map $$\gamma_{\cdot,N}:\left((0,1],\mathcal B,\mu_R\right)\rightarrow\left(\mathcal C([0,1],\C),\mathcal B_{\mathcal C}\right),$$
where $\mathcal B_{\mathcal C}$ is the Borel $\sigma$-algebra on $\mathcal C([0,1],\C)$ coming from the topology 
of uniform convergence.
Let $\mathrm P_N$ be the corresponding induced probability measure on $\mathcal C([0,1],\C)$, $\mathrm P_N(A):=\mu_R\!\left(\gamma_{\cdot,N}^{-1}(A)\right)$, where $A\in\mathcal B_{\mathcal{C}}$. 
For  $0\leq t_1<t_2<\cdots<t_k\leq1$, let $\pi_{t_1,\ldots,t_k}:\mathcal{C}([0,1],\C)\rightarrow\C^k$ be the natural projection defined as $\pi_{t_1\,\ldots,t_k}(\gamma):=(\gamma(t_1),\ldots,\gamma(t_k))$.

Theorem \ref{main-theorem} can be rephrased as follows: 
for every $k\in\N$ and for every $0\leq t_1<\ldots<t_k\leq1$
\begin{equation}\nonumber
\mathrm P_N\pi^{-1}_{t_1,\ldots,t_k}\Longrightarrow\mathrm P^{(k)}_{t_1,\ldots,t_k}\hspace{.4cm}\mbox{as $N\rightarrow\infty$},
\end{equation}
where \virg{$\Rightarrow$} denotes weak convergence of probability measures. In other words, 
we prove weak convergence of finite-dimensional distributions of $\mathrm P_N$ as $N\rightarrow\infty$.
\end{remark}
\begin{remark}
By construction, the measures $\mathrm P_{t_1,\ldots,t_k}^{(k)}$ automatically satisfy Kolmogorov's consistency conditions and hence there exists 
a probability measure $\tilde{\mathrm P}$ on the $\sigma$-algebra generated by finite dimensional cylinders $\mathcal B_{\tiny{\mbox{fdc}}}\subset\mathcal B_{\mathcal C}$ so that $\tilde{\mathrm P}\pi_{t_1,\ldots,t_k}^{-1}=\mathrm P_{t_1,\ldots,t_k}^{(k)}$.
\end{remark}
\begin{remark}[Scaling property of the limiting measures]
Notice that $$\gamma_{\alpha,N}(\lambda t)=N^{-\ha}\su{\alpha}{\lambda t N}=\lambda^{\ha}\gamma_{\alpha,\lambda N}(t).$$ Thus, 
the limiting probability measures $\mathrm P^{(k)}_{t_1,\ldots,t_k}$ satisfy the following scaling property: for every $\lambda\in(0,1]$
\begin{equation}\label{scaling property}\nonumber
\mathrm P^{(k)}_{\lambda t_1,\ldots,\lambda t_k}(A)=\mathrm P^{(k)}_{t_1,\ldots,t_k}(\lambda^{-\ha}A)
\end{equation}
In particular, for example, $\mathrm P^{(1)}_{t}(A)=\mathrm P^{(1)}_1(t^{-\ha}A)$.
\end{remark}
\begin{remark}
Our results are of \emph{probabilistic} nature, since we look at the measure of $\alpha$'s for which some event happens. Let us stress the fact that the growth of $|\su{\alpha}{N}|$ for \emph{specific} or \emph{generic} $\alpha$ has also been thoroughly studied. For instance, Hardy and Littlewood \cite{Hardy-Littlewood1914} proved that if $\alpha$ is of bounded-type, then $|\su{\alpha}{N}|\leq C\sqrt N$ for some constant $C$. To the best of our knowledge, the most refined result in this direction is due to Flaminio and Forni \cite{Flaminio-Forni2006}. A particular case of their results on equidistribution of nilflows reads as follows. For every increasing function $b:(1,\infty)\rightarrow(0,\infty)$ such that $\int_1^\infty t^{-1}b^{-4}(t)\de t<\infty$, there exists a full measure set $\mathcal G_b$ such that for every $\alpha\in\mathcal{G}_b$, every $\beta\in\R$ the following holds: for every $s>\frac{5}{2}$, there exists a constant $C=C(s,\alpha)$ such that for every $f\in W^s$, $2$-periodic, $$\left|\sum_{n=0}^{N-1}f(\alpha n^2+\beta)-N\int_{-1}^1f(x)\de x\right|\leq C\sqrt N\, b(N)\|f\|_s,$$
where $W_s$ denotes the Sobolev space and $\|\cdot\|_s$ is the corresponding Sobolev norm. This generalizes the work of Fiedler, Jurkat and K\"{o}rner \cite{Fiedler-Jurkat-Korner77} where $f(x)=e^{\pi i x}$ and $\beta=0$.
\end{remark}

The paper is organized as follows. In Section \ref{section: renormalization-of-curlicues} we discuss the geometric multi-scale structure of the curve $t\mapsto\gamma_{\alpha,N}(t)$ and we deal with the first step of the renormalization procedure which allows us to move from a scale to the next one. Moreover, we describe the connection of the renormalization map $T$ with the continued fraction expansion of $\alpha$ with even partial quotients and we consider an \virg{accelerated} version of it, i.e. the associated jump transformation $R$. For the corresponding accelerated continued fraction expansions we prove some estimates on the growth of the entries.
In Section \ref{iterated-renormalization-of-S} we iterate the renormalization procedure and we approximate the curve $\gamma_{\alpha,N}$ by a curve $\gamma_{\alpha,N}^J$ in which only the $J$ largest scales are present. Furthermore, we write $(\gamma_{\alpha,N}(t_j))_{j=1}^k\in\C^k$ as a function of certain random variables defined in terms of the renewal time $\hat n_N:=\min\{n\in\N:\:\hat q_n>N\}$, where $\{\hat q_n\}_{n\in\N}$ is the subsequence of denominators of the convergents of $\alpha$ corresponding to the map $R$. 
In Section \ref{section-limiting-fin-dim-distr} we use a renewal-type limit theorem (proven in Appendix \ref{AppendixA}) to show the existence of the limit for finite-dimensional distributions for the approximating curve $\gamma_{\alpha,N}^J$ as $N\rightarrow\infty$. Estimates from Section \ref{iterated-renormalization-of-S} allow us to take the limit as $J\rightarrow\infty$ and prove the existence of finite-dimensional distributions for $\gamma_{\alpha,N}$ as $N\rightarrow\infty$. We also discuss the notion of \emph{nice} sets and give a sufficient condition for a set to be nice.

\section{Renormalization of Curlicues}\label{section: renormalization-of-curlicues}
In this section we recall some known facts concerning the geometry of the curves $\gamma_{\alpha, N}$. In particular we discuss the presence/absence of spiral-like fragments and at different scales using a renormalization procedure. The renormalization map $T$ is connected with a particular class of continued fraction expansions. From a metrical point of view, this classical renormalization is very ineffective, because of the intermittent behavior of the map $T$ (which preserves an infinite, ergodic measure). It is therefore very natural to study an ``accelerated version'' of $T$ (preserving the ergodic probability measure $\mu_R$ mentioned before) and the corresponding continued fraction expansion.

\subsection{Geometric structure at level zero}
In order to investigate the presence/absence of spiraling geometric structures at the smallest scale we introduce the \emph{local discrete radius of curvature}, following Coutsias and Kazarinoff \cite{Coutsias-Kazarinoff1987,Coutsias-Kazarinoff-1998}. Set $\mathcal{T}_N:=\left\{\frac{m}{N},\,0\leq m\leq N\right\}$ and let 
$\tau_n:=\frac{n}{N}\in\mathcal{T}_N\smallsetminus\{0,1\}$, so that 
$\gamma(\tau_n)=\gamma_{\alpha,N}(\tau_n)=N^{-\ha}\su{\alpha}{n}$. 
Define $
\rho_{\alpha,N}(\tau_n)$ as the radius of the circle passing through the three points $\gamma(\tau_{n-1})$, $\gamma(\tau_{n})$ and $\gamma(\tau_{n+1})$. 
A simple computation shows that
$\rho_{\alpha,N}(\tau_n)=\frac{1}{2 \sqrt{N}}\left|\csc\left(\frac{\pi\,\alpha\,(2n-1)}{2}\right)\right|$ and for arbitrary $t\in[0,1]$ we set $$\rho(t)=\rho_{\alpha,N}(t):=\frac{1}{2 \sqrt{N}}\left|\csc\left(\frac{\pi\,\alpha\,(2\,tN-1)}{2}\right)\right|\in\overline\R.$$
The function $t\mapsto\rho_{\alpha,N}(t)$ is $\frac{1}{\alpha\, N}$-periodic; it has vertical asymptotes at $\tau_k^\fl=\tau_{k}^\fl(\alpha,N):=\frac{k}{\alpha\,N}+\frac{1}{2N}$ and local minima at $\tau_{k}^\cu=\tau_{k}^\cu(\alpha,N):=\frac{2k+1}{2\,\alpha\,N}+\frac{1}{2N}$, $k\in\Z$, where $\rho_{\alpha,N}(\tau_{k}^\cu)=\frac{1}{2\sqrt{N}}$.
We partition the interval $[0,1]$ into subintervals as follows: $$[0,1]=\bigsqcup_{k=0}^{k^*+1} I_{k}^{(0)},$$ where $k^*=k^*_{\alpha,N}:=\left\lfloor \alpha\,N-\frac{\alpha+1}{2}\right\rfloor$ and
\begin{equation}\nonumber
I_{k}^{(0)}=I_{k;\alpha,N}^{(0)}:=\begin{cases}
\left[0,\tau_{0}^\cu\right)  & \text{if $k=0$}, \\&\\
\left[\tau_{k-1}^\cu,\tau_{k}^\cu\right)      & \text{if $1\leq k\leq k^*$},\\&\\
\left[\tau_{k^*}^\cu,1\right] & \text{if $k=k^*+1$}.
\end{cases}\end{equation}
By construction, the lengths of the above intervals are $|I_k^{(0)}|=\frac{1}{\alpha\,N}$ for $1\leq k\leq k^*$
, $|I_0^{(0)}|=\frac{1}{2N}$ and $0\leq|I_{k^*+1}^{(0)}|=1-\frac{1}{2N}-\frac{k^*}{\alpha\,N}<\frac{1}{\alpha N}$.
The number of $\mathcal{T}_N$-rationals inside each subinterval is of order $\frac{1}{\alpha}$ and explicitly given by 
\begin{equation}\nonumber
\#(I_k^{(0)}\cap\mathcal{T}_N)=\begin{cases}
\left\lceil \frac{1}{2\alpha}+\ha\right\rceil
& \text{if $k=0$},\\&\\
\left\lceil\frac{2k+1}{2 \alpha}+\ha\right\rceil-\left\lceil\frac{2k-1}{2 \alpha}+\ha\right\rceil & \text{if $1\leq k\leq k^*$},\\&\\
N+1-\left\lceil\frac{2k^*+1}{2\alpha}+\ha\right\rceil&\text{if $k=k^*+1$}.
\end{cases}
\end{equation}
%
The whole curve $\gamma_{\alpha,N}([0,1])$ can be recovered by means of the values of the function $\rho$ 
at the rationals in $\mathcal{T}_N$. 
Suppose we know the values of $\gamma(\tau_0),\gamma(\tau_1),\ldots,\gamma(\tau_{n-1}),\gamma(\tau_n)$ and the radius $\rho
(\frac{n}{N})$. Then the point $\gamma(\tau_{n+1})$ should be placed at the intersection of the circle of radius $N^{-\ha}$ centered at $\gamma(\tau_n)$ and one of the two circles of radius $\rho
(\frac{n}{N})$ passing through $\gamma(\tau_{n-1})$ and $\gamma(\tau_n)$ 
in order to get a counterclockwise oriented triple $(\gamma(\tau_{n-1}),\gamma(\tau_n),\gamma(\tau_{n+1}))$ when $\frac{n}{N}\in[\tau_{k-1}^\cu,\tau_k^\fl)$ (resp. clockwise when $\frac{n}{N}\in[\tau_k^\fl,\tau_{k}^\cu)$). For arbitrary $t\in[0,1]$ the curve $\gamma(t)$ is defined by linear interpolation.
\begin{figure}[h!]
\begin{center}
\vspace{-4cm}
\includegraphics[width=12.5cm, angle=0]{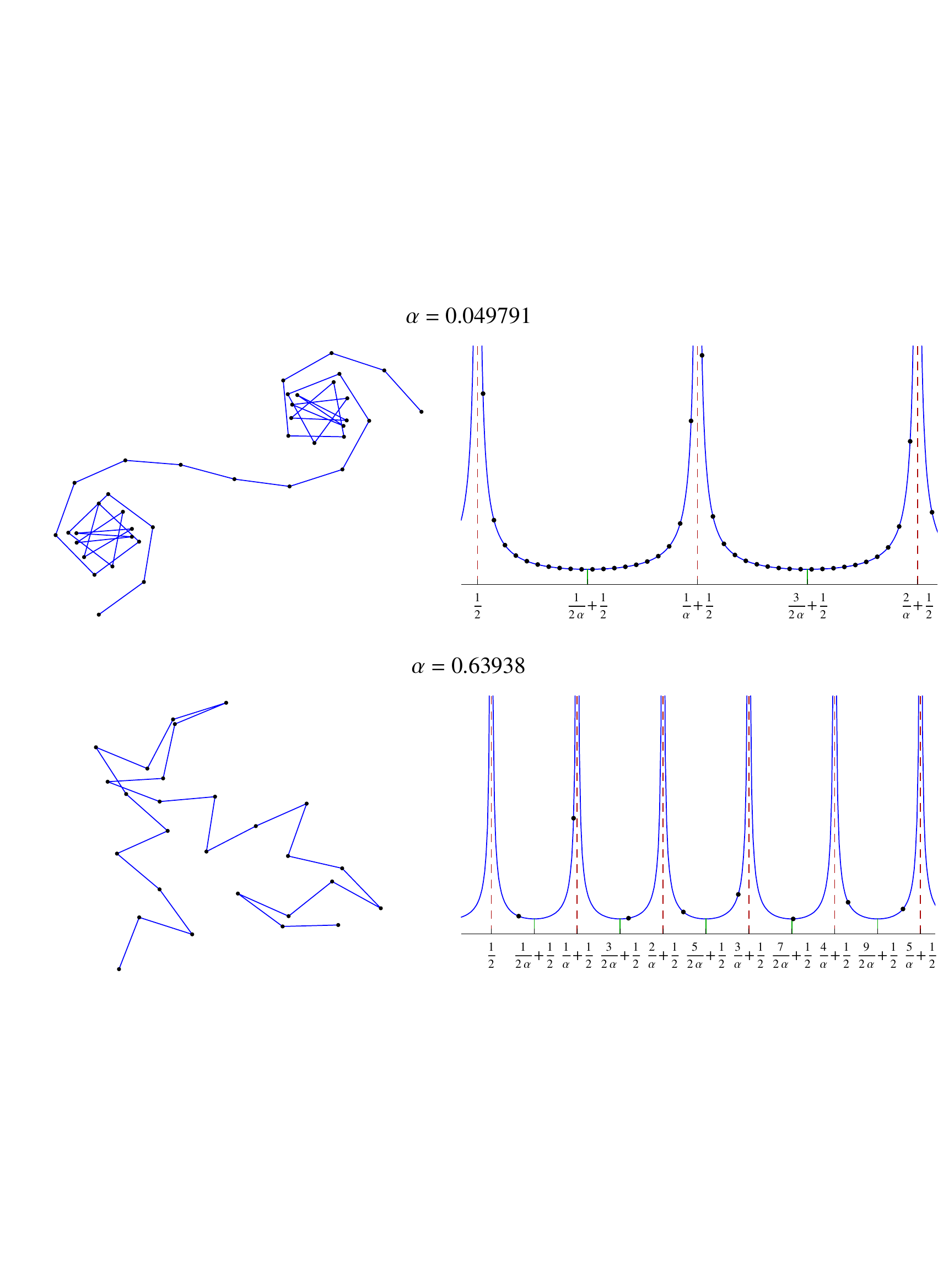}
\vspace{-4.3cm}
\caption{\small{Geometric patterns at level zero (left) and the function $\rho_{\alpha,N}$ (right).}}\label{fig2}
\end{center}
\end{figure}
For small values of $\alpha$, each subinterval $I^{(0)}_k$, $1\leq k\leq k^*$, 
contains approximately $\frac{1}{\alpha}$ integer multiples of $\frac{1}{N}$ and the curlicue structure is easily understood: those $n$'s for which $\rho(\tau_n)$ is large correspond to straight-like parts of $\gamma([0,1])$, 
while the points close to the minima of $\rho$ give the spiraling fragments (\emph{curlicues}). For $\alpha\sim1$ the curlicues disappear. See Figure \ref{fig2}. We shall see in Section \ref{AERFs} how these curlicues appear at different scales though.

\subsection{Approximate and Exact Renormalization Formul\ae}\label{AERFs}
Let us introduce the map $U:(-1,1]\smallsetminus\{0\}\rightarrow (-1,1]\smallsetminus\{0\}$ where $U(t):=-\frac{1}{t}$ (mod 2). The graph of $U$ 
has countably many smooth branches. Each interval $\left(\frac{1}{2k+1},\frac{1}{2k-1}\right]$ is mapped in a one-to-one way onto $(-1,1]$ via $t\mapsto-\frac{1}{t}+2k$.

For $a\in(-1,1]\smallsetminus\{0\}$ and $N\in\N$ one has the Approximate Renormalization Formula (ARF)
\begin{equation}\label{arf1}
\left|\su{a}{N}-e^{\frac{\pi}{4}i}\,|a|^{-\ha}\,\su{a_1}{\lfloor N_1\rfloor}\right|\leq C_1|a|^{-\ha}+C_2,
\end{equation}
where $a_1=U(a)$, $N_1=
|a|\,N
$ and $C_1, C_2>0$ are absolute constants which do not depend on $N$. This result was established by Hardy and Littlewood \cite{Hardy-Littlewood1914}, Mordell \cite{Mordell1926}, Wilton \cite{Wilton1926} and Coutsias and Kazarinoff \cite{Coutsias-Kazarinoff-1998}, 
the constants $C_1,C_2$ being always improved. 

Let us explain the ARF (\ref{arf1}) geometrically. Recall that the curve $t\mapsto\gamma_{a,N}(t)$ contains $k^*_{|a|,N}\simeq N_1$ intervals of the form $[\tau_{k-1}^\cu,\tau_{k}^\cu)$ at level zero. By (\ref{arf1}), the curve $t\mapsto\sqrt N\gamma_{a,N}(t)$ can be approximated (up to scaling by $|a|^{-\ha}$ and rotating by $\frac{\pi}{4}$) by $t\mapsto\sqrt{N_1}\gamma_{a_1,N_1}(t)$. In other words, replace each interval of the form $I^{(0)}_k$, $1\leq k\leq k^*$, for $\gamma_{a,N}(t)$ by a $\mathcal T_{N_1}$-rational point in $\gamma_{a_1,N_1}(t)$. The renormalization map can be seen as a \virg{coarsening} transformation, which deletes of the geometric structure at level zero.
Beside the above-mentioned references, we also want to mention the work by Berry and Goldberg \cite{Berry-Goldberg-1988}, in which typical and untypical behaviors of $\{\su{\alpha}{N'}\}_{N'=1}^{N}$ are studied with the help of a renormalization procedure. 

Coutsias and Kazarinoff \cite{Coutsias-Kazarinoff-1998} also proved a stronger version of (\ref{arf1}):
\begin{equation}\nonumber
\left|\su{a}{N}-e^{\pifi}\,|a|^{-\ha}\,\su{a_1}{n}\right|\leq C_3\left|\frac{|a| N-n}{a}\right|\leq C_4,
\end{equation}
for some $C_3,C_4>0$, where $n\in\N$ is arbitrary and $N=\left\langle n/|a|\right\rangle$ is a function of $n$, $\langle\cdot\rangle$ denoting the nearest-integer function.

In our analysis we shall focus on (\ref{arf1}), which can be extended to 
$\su{a}{L}$ for arbitrary $L\geq0$:
\begin{equation}\label{arf2}
\left|\su{a}{L}-e^{\frac{\pi}{4}i}\,|a|^{-\ha}\,\su{a_1}{L_1}\right|\leq C_5|a|^{-\ha}+C_6,
\end{equation}
where $a_1=U(a)$, $L_1=|a|\,L$, $C_5=C_1+2$ and $C_6=C_2+1$. 

Since the function $U$ is odd w.r.t the origin and $\su{-a}{N}=\overline{\su{a}{N}}$, it is natural to consider $\alpha=|a|\in(0,1]$ and keep track of $|U(\alpha)|$ and $\mathrm{sgn}(U(\alpha))$ separately. Define $\eta(\alpha):=\mathrm{sgn}\left(U(\alpha)\right)$, $\xi(\alpha):=-\eta(\alpha)$ and introduce
a new map $T:(0,1]\rightarrow(0,1]$, $T:=\left|U|_{(0,1]}\right|$. More explicitly, let us partition the interval $(0,1]$ into subintervals $B(k,\xi)$, $k\in\N$, $\xi=\pm1$, where 
$B(k,-1):=\left(\frac{1}{2k},\frac{1}{2k-1}\right]$ and $B(k,+1):=\left(\frac{1}{2k+1},\frac{1}{2k}\right]$. The map $T$ can be represented accordingly as
$$
T(\alpha)=\xi\cdot\left(\frac{1}{\alpha}-2k\right),\hspace{.8cm}\alpha\in B(k,\xi),\:\: k\in\N,\: \:\xi\in\{\pm1\}.$$
We shall deal with this map, first introduced by Schweiger \cite{Schweiger82,Schweiger84}, in Section \ref{ECF} in connection with the even continued fraction expansion of $\alpha$. 
Moreover, for every complex-valued function $F$ 
set $$F^{(\eta)}:=\begin{cases}
      F& \text{if $\eta=+1$}, \\
      \overline{F}& \text{if $\eta=-1$}.
\end{cases}$$
With this notations we can define the remainder terms of (\ref{arf1}) and (\ref{arf2}) for $\alpha\in(0,1]$ as follows:
\begin{eqnarray}
&&\Lambda(\alpha,N):=\su{\alpha}{N}-e^{\pifi}\,\alpha^{-\ha}\s{\alpha_1}{\eta_1}{\lfloor N_1\rfloor},\hspace{1cm}N\in\N
\label{Lambda}\\
&&\Gamma(\alpha,L):=\su{\alpha}{L}-e^{\pifi}\,\alpha^{-\ha}\s{\alpha_1}{\eta_1}{L_1},\hspace{1cm}L\in\R\label{Gamma}
\end{eqnarray}
where
$\alpha_1=T(\alpha)$, $\eta_1=\eta(\alpha)$, $N_1=\alpha\,N$ and $L_1=\alpha\,L$.

Later, we shall use the fact that $\Gamma(\alpha,L)$ is a continuous function of $(\alpha,L)\in(0,1]\times\R_{\geq0}$ (one can actually prove that it has piecewise $\mathcal C^\infty$ partial derivatives).
An explicit formula for $\Lambda(\alpha,N)$, $N\in\N$, has been provided by Fedotov and Klopp \cite{Fedotov-Klopp} in terms of a special function $\mathcal{F}_\alpha:\C\rightarrow\C$ as follows. For $\alpha\in(0,1]$ and $w\in\C$ set 
\begin{equation}\label{special F Fedotov-Klopp}\mathcal{F}_\alpha(w):=\int_{\Gamma_w}\frac{\exp\left(\pi\,i\,z^2/\alpha\right)}{\exp\left(2\pi\, i\,(z-w)\right)-1}\,\de z,
\end{equation}
where $\Gamma_w$ is the contour given by $$\R\ni t\mapsto\Gamma_w(t)=\begin{cases}w+t+i\,t&\text{if $|t|\geq \varepsilon$,}\\w+\varepsilon \,\exp\left(\pi\,i\left(\frac{t}{2\varepsilon}-\frac{1}{4}\right)\right)&\text{if $|t|<\varepsilon$,}\end{cases}$$
and $\varepsilon=\varepsilon(\alpha,w)$ is 
smaller than the distance between $w$ and the other poles of the integrand in (\ref{special F Fedotov-Klopp}).
We have the following
\begin{theorem}[Exact Renormalization Formula, Fedotov-Klopp, \cite{Fedotov-Klopp}] For every $0<\alpha\leq1$ and every $N\in\N$ we have
\begin{equation}\label{ERF-Fedotov-Klopp}
\overline{\Lambda(\alpha,N)}=e^{-\pifi}\,\alpha^{-\ha}\left[e^{
-\pi\,i\,
\alpha\,N^2 
}\,\mathcal{F}_\alpha(\{N_1\})-\mathcal{F}_\alpha(0)\right],
\end{equation}
where $N_1=\alpha\,N$.
\end{theorem}
In order to write $\Gamma(\alpha,L)$ in terms of $\Lambda(\alpha,\lfloor L\rfloor)$, we notice that $\alpha\,L=\lfloor\alpha\lfloor L\rfloor\rfloor+H(\alpha,L)$, where $H(\alpha,L):=\alpha\{L\}+\{\alpha\lfloor L\rfloor\}\in[0,2)$. Moreover, if $
H(\alpha,L)\in[0,1)$ then $\lfloor\alpha\,L\rfloor=\lfloor\alpha\lfloor L\rfloor\rfloor$, while if $
H(\alpha,L)\in[1,2)$ then $\lfloor\alpha\,L\rfloor=\lfloor\alpha\lfloor L\rfloor\rfloor+1$.
Now,
a simple computation shows that for every $\alpha\in(0,1]$ and every $L\geq0$
\begin{eqnarray}
\Gamma(\alpha,L)=\Lambda(\alpha,\lfloor L\rfloor)+G_1(\alpha,L)-e^{\pifi}\,\alpha^{-\ha}\,G_2(\alpha,L),\label{Gamma=Lambda+err}
\end{eqnarray}
where $G_1(\alpha,L):=\{L\}\,e^{
\pi\,i\,\lfloor L\rfloor^2\alpha
}
$ and
\begin{eqnarray}
G_2(\alpha,L):=\begin{cases}
H(\alpha,L)\,e^{\pi\,i\,\lfloor\alpha L\rfloor^2\alpha_1}&\text{if $H(\alpha,L)\in[0,1)$},\\
e^{\pi\,i\,\left(\lfloor\alpha L\rfloor-1\right)^2\alpha_1}+(H(\alpha,L)-1)\,e^{\pi\,i\,\lfloor\alpha L\rfloor^2\alpha_1}&\text{if $H(\alpha,L)\in[1,2)$}.
\end{cases}\nonumber
\end{eqnarray}

\begin{remark}
Applying the stationary phase method to the integrals in (\ref{ERF-Fedotov-Klopp}) and (\ref{Gamma=Lambda+err}) as in \cite{Fedotov-Klopp} one can obtain the approximate renormalization estimates (\ref{arf1}) and (\ref{arf2}) (possibly with different constants $C_1$, $C_2$, $C_5$, $C_6$).
\end{remark}
We want to describe $\su{\alpha}{t\,N}$ for $N\in\N$ and $t\in[0,1]$. In this case (\ref{arf2}) and (\ref{Gamma=Lambda+err}) can be rewritten as  
\begin{eqnarray}
&\su{\alpha}{t\,N}=e^{\frac{\pi}{4}i}\,\alpha^{-\ha}\,\s{\alpha_1}{\eta_1}{t\,\alpha\,N}+\Gamma(\alpha,t\,N),\label{erf1}\\
&
\Gamma(\alpha,t\,N)=\Lambda(\alpha,\lfloor t\,N\rfloor)+G_1(\alpha,t\,N)+e^{\pifi}\,\alpha^{-\ha}\,G_2(\alpha,t\,N).\label{erf2}
\end{eqnarray}

\subsection{Continued Fractions With Even Partial Quotients}\label{ECF}
In this section we discuss the relation between the map $T$ and expansions in continued fractions with even partial quotients. 
Consider the following \emph{ECF-expansion} for $\alpha\in (0,1]$:
\begin{equation}
\alpha=\frac{1}{2k_1+\frac{\xi_1}{2k_2+\frac{\xi_2}{2k_3+\cdots}}}=:\les(k_1,\xi_1),(k_2,\xi_2),(k_3,\xi_3),\ldots\res,
\end{equation}
where $k_j\in\N$ and $\xi_j\in\{\pm1\}$, $j\in\N$. ECF-expansions have been introduced by Schweiger \cite{Schweiger82, Schweiger84} and studied by Kraaikamp-Lopes \cite{Kraaikamp-Lopes96}. Since $1=\les(1,-1),(1,-1),\ldots\res$, it is easy to see that every $\alpha\in\X$ has an infinite expansion with no $(1,-1)$-tail.
 
Using the notations introduced in Section \ref{AERFs} we notice that if $\alpha\in B(k,\xi)$, then $\alpha=\frac{1}{2k+\xi\,T(\alpha)}$. Therefore, 
\begin{eqnarray}
&&\mbox{for}\hspace{.5cm}\alpha=\les(k_1,\xi_1),(k_2,\xi_2),(k_3,\xi_3),\ldots\res\in B(k_1,\xi_1), \nonumber\\
&&
T^n(\alpha)=\les(k_{n+1},\xi_{n+1}),(k_{n+2},\xi_{n+2}),\ldots\res\in B(k_{n+1},\xi_{n+1}),
\end{eqnarray}
i.e. $T$ acts as a shift on the space $\Omega^\N$, where $\Omega:=\N\times\{\pm1\}$. Despite its similarities with the Gauss map in the context of Euclidean continued fractions, the map $T$ has an indifferent fixed point at $\alpha=1$ and we have the following
\begin{theorem}[Schweiger, \cite{Schweiger82}]
The map $T:(0,1]\rightarrow(0,1]$ has a $\sigma$-finite, infinite, ergodic invariant measure $\mu_T$ which is absolutely continuous w.r.t. the Lebesgue measure on $(0,1]$. Its density is $\varphi_T(\alpha):=\frac{\de \mu_T(\alpha)}{\de \alpha}=\frac{1}{\alpha+1}-\frac{1}{\alpha-1}$.
%
\end{theorem}
One of the consequences of this fact is the anomalous growth of Birkhoff sums for integrable functions. 
Given $f\in L^1\left((0,1],\mu_T\right)$, $f\geq0$ $\mu_T$-almost everywhere, let $\mu_T(f)=\int_0^1f(\alpha)\,\de\mu_T(\alpha)$ and denote by $\mathrm{S}_n^T(f)$ the ergodic sum $\sum_{j=0}^{n-1}f\circ T^j$. 
Since $\mu_T((0,1])=\infty$, the Birkhoff Ergodic Theorem implies that $\frac{1}{n}\mathrm{S}_n^T(f)\rightarrow 0$ almost everywhere as $n\rightarrow\infty$. According to the Hopf's Ergodic Theorem there exists a sequence of measurable functions $\left\{a_n(\alpha)\right\}_{n\in\N}$ such that $\frac{1}{a_n(\alpha)}\mathrm{S}_n^T(f)(\alpha)\rightarrow\mu_T(f)$ for almost every $\alpha\in(0,1]$ as $n\rightarrow\infty$. The question \emph{\virg{Can the sequence $a_n(\alpha)$ be chosen independently of $\alpha$?}} is answered negatively by Aaronson's Theorem (\cite{AaronsonBook}, Thm. 2.4.2), according to which for almost every $\alpha\in(0,1]$ and for every sequence of constants $\{a_n\}_{n\in\N}$ either $\liminf_{n\rightarrow\infty}\frac{1}{a_n}\mathrm{S}_n^T(f)(\alpha)=0$ or $
\frac{1}{a_{n_k}}\mathrm{S}_{n_k}^T(f)(\alpha)\rightarrow\infty$ along some subsequence $\{a_{n_k}\}_{k\in\N}$ as $k\rightarrow\infty$. However, for weaker types of convergence such a sequence of constants can indeed be found. The following Theorem establishes $a_n=\frac{n}{\log n}$ and provides convergence in probability:
%
\begin{theorem}[Weak Law of Large Numbers for $T$]\label{WLLNforT}
For every probability measure $\mathrm{P}$ on $(0,1]$, absolutely continuous w.r.t. $\mu_T$, for every $f\in L^1(\mu_T)$ and for every $\varepsilon>0$,
$$\mathrm{P}\!\left(\left|\frac{\mathrm{S}_n^T(f)}{\frac{n}{\log n}}-\mu_T(f)\right|\geq\varepsilon\right)\longrightarrow 0\hspace{1cm}\mbox{as $n\rightarrow\infty$.}$$
\end{theorem}
\begin{remark}
The proof of Theorem \ref{WLLNforT} follows from standard techniques in infinite ergodic theory. See Aaronson  \cite{Aaronson-f-expansions} and \cite{AaronsonBook}, \S 4.
The same rate $\frac{n}{\log n}$ rate for the growth of Birkhoff sums for integrable observables over ergodic transformations preserving an infinite measure appears in several examples, e.g. the Farey map. A recent interesting example comes from the study of linear flows over regular $n$-gons, see Smillie and Ulcigrai \cite{Ulcigrai-Smillie}. 
\end{remark}


Let us come back to ECF-expansions. For $\alpha=\les(k_1,\xi_1),(k_2,\xi_2),\ldots\res$ the convergents have the form 
$$\frac{p_n}{q_n}=\frac{1}{2k_1+\frac{\xi_1}{2k_2+\frac{\xi_2}{2k_3+\cdots+\frac{\xi_{n-2}}{2k_{n-1}+\frac{\xi_{n-1}}{2k_n}}}}}=\les(k_1,\xi_1),(k_2,\xi_2),\ldots,(k_n,*)\res,\hspace{.5cm}(p_n,q_n)=1,$$
where \virg{$*$} denotes any $\xi_n=\pm1$.
They satisfy the following recurrent relations:
\begin{eqnarray}
p_n=2k_n\,p_{n-1}+\xi_{n-1}\,p_{n-2},\hspace{1cm}q_n=2k_n\,q_{n-1}+\xi_{n-1}\,q_{n-2}\label{recurrent-relations-p-q},
\end{eqnarray}
with $q_{-1}=p_0=0$, $p_{-1}=q_0=\xi_0=1$. 
Moreover, we have
\begin{equation}\label{formula pnqnp1-pnp1qn}
p_{n+1}q_{n}-p_{n}q_{n+1}=(-1)^n\prod_{j=0}^n\xi_j.
\end{equation}
The proof of (\ref{formula pnqnp1-pnp1qn}) follows from (\ref{recurrent-relations-p-q}) and
can be recovered \emph{mutatis mutandis} from the proof of the analogous result for Euclidean continued fractions. See, e.g., \cite{Rockett-Szusz}.

Set $\alpha_0:=\alpha$ and $\alpha_n:=T^n(\alpha)$. In Section \ref{iterated-renormalization-of-S}, we shall deal with the product $\alpha_0\alpha_1\cdots\alpha_{n-1}$. As in the case of Euclidean continued fractions, this product can be written in terms of the denominators of the convergents; however the formula involves the $\xi_n$ as well: for $n\in\N$,
\begin{equation}\label{prodalpha}
(\alpha_0\cdots\alpha_{n-1})^{-1}=q_n\left(1+\xi_n\,\alpha_n\frac{q_{n-1}}{q_n}\right).
\end{equation}
Notice that, considering $f(\alpha)=-\log{\alpha}$, Theorem \ref{WLLNforT} reads as follows: for every $\varepsilon>0$ and every probability measure $\mathrm P$ on $(0,1]$, absolutely continuous w.r.t. $\mu_T$, $$\mathrm P\!\left(\left|\frac{-\log(\alpha_0\cdots\alpha_{n-1})}{\frac{n}{\log n}}-\frac{\pi^2}{4}\right|\geq\varepsilon\right)\rightarrow0\hspace{.5cm}\mbox{as $n\rightarrow\infty$}.$$
In other words, the product along the $T$-orbit of $\alpha$ decays subexponentially in probability.
 
\subsection{The Jump Transformation $R$}\label{section: jump transformation}
In order to overcome the issues connected with the infinite invariant measure for $T$, it is convenient to introduce an \virg{accelerated} version of $T$, namely its associated \emph{jump transformation} (see \cite{Schweiger_ET_Fibred_Systems}) $R:(0,1]\rightarrow(0,1]$. Define the \emph{first passage time} to the interval $\left(0,\frac{1}{2}\right]$ as $\tau:(0,1]\rightarrow\N_0=\N\cup\{0\}$ as $\tau(\alpha):=\min\left\{j\geq0:\:T^j(\alpha)\in B(1,-1)^c=\left(0,\frac{1}{2}\right]\right\}$ and the jump transformation w.r.t. $\left(0,\frac{1}{2}\right]$ as $R(\alpha):=T^{\tau(\alpha)+1}(\alpha)$. Let us remark that this construction is very natural. For instance, if we consider the jump transformation associated to the Farey map w.r.t. the interval $(\frac{1}{2},1]$ we get precisely the celebrated Gauss map. Another example is given by the Zorich map, obtained by accelerating the Rauzy map, in the context of interval exchange transformations. 

The map $R$ was extensively studied 
in \cite{Cellarosi}. It is a Markov, uniformly expanding map with bounded distortion and has an invariant probability measure $\mu_R$ which is absolutely continuous w.r.t. the Lebesgue measure on $[0,1]$. The density of $\mu_R$ is given by $\varphi_R(\alpha):=\frac{\de\mu_R(\alpha)}{\de \alpha}=\frac{1}{\log 3}\left(\frac{1}{3-\alpha}+\frac{1}{1+\alpha}\right)$. For a different acceleration of $T$ in connection with the geometry of theta sums, see Berry and Goldberg \cite{Berry-Goldberg-1988}.

We want to describe a symbolic coding for $R$. Let us restrict ourselves to $\alpha\in\X$ and identify $\X$ with the subset $\dot{\Omega}^N\subset\Omega^\N$ of infinite sequences with no $(1,-1)$-tail. Let $\bar\omega=(1,-1)$. Given $\alpha=\les\omega_1,\omega_2,\omega_3,\ldots\res\in\dot{\Omega}^\N$ we have $\tau=\tau(\alpha)=\min\{j\geq0:\:\omega_{j+1}\neq\bar\omega\}$ and $R(\alpha)=\les\omega_{\tau+2},\omega_{\tau+3},\omega_{\tau+4},\ldots\res\in\dot\Omega^\N$. Setting $\Omega^*:=\Omega\smallsetminus\{\bar\omega\}$, $\Sigma:=\N_0\times\Omega^*$ and denoting by $\sigma=(h,\omega)\in\Sigma$ the $\Omega$-word $(\bar\omega,\ldots,\bar\omega,\omega)$ of length $h+1$ for which $\omega\in\Omega^*$, we can identify $\dot{\Omega}^\N$ and $\Sigma^\N$ and the map $R$ acts naturally as a shift over this space.
%
 
For brevity, we denote $m^\pm=0\cdot m^\pm=(0,(m,\pm1))\in\Sigma$ and $h\cdot m^\pm=(h,(m,\pm1))\in\Sigma$. For $\alpha=(h_1\cdot m_1^\pm,h_2\cdot m_2^\pm,\ldots)\in\Sigma^\N$ define $\nu_0:=1$, $\nu_n=\nu_n(\alpha)=h_1+\ldots+h_n+n+1$ 
and let $\hat q_n=\hat q_n(\alpha):=q_{\nu_n(\alpha)}(\alpha)$ be the denominator of the $n$-th $R$-convergent of $\alpha$. We shall refer to $\{\hat q_n\}_{n\in\N}$ as \emph{$R$-denominators} and to $(h_j\cdot m_j^\pm)$ as \emph{$\Sigma$-entries}.

In \cite{Cellarosi} the following estimates were proven:
\begin{lem}\label{growth of R-denominators}
\begin{itemize}
\item[(i)] For every $\alpha\in (0,1]$, $\hat q_n\geq 3^{n/3}$.
\item[(ii)] For Lebesgue-almost every $\alpha\in(0,1]$ and sufficiently large $n$, $\hat q_n\leq e^{C_{7} n}$, where $C_7>0$ is some constant.
\end{itemize}
\end{lem}

In Section \ref{iterated-renormalization-of-S}, we will need the following renewal-type limit theorem.
\begin{theorem}\label{renewal-type-limit-theorem-for-R}
Let $L>0$ and $\hat n_L=\hat n_L(\alpha)=\min\{n\in\N:\:\hat q_n>L\}$. Fix $N_1,N_2\in\N$. The ratios $\frac{\hat q_{\hat n_L-1}}{L}$ and $\frac{\hat q_{\hat n_L}}{L}$ and the entries $\sigma_{\hat n_L+j}$, $-N_1<j\leq N_2$ have a joint limiting probability distribution w.r.t. the measure $\mu_R$ as $L\rightarrow\infty$. 

In other words, there exists a probability measure $\mathrm Q^{(0)}=\mathrm Q^{(0)}_{N_1,N_2}$ on the space $(0,1]\times(1,\infty)\times\Sigma^{N_1+N_2}$ such that for every $0\leq a<b\leq1\leq c<d$ and every $(N_1+N_2)$-tuple $\underline\vartheta=\{\vartheta_j\}_{j=-N_1+1}^{N_2}\in\Sigma^{N_1+N_2}$ we have 
\begin{eqnarray}
&&\lim_{L\rightarrow\infty}\mu_R\left(\left\{\alpha:\:\:a<\frac{\hat q_{\hat n_L-1}}{L}< b,\:\:c<\frac{\hat q_{\hat n_L}}{L}< d,\:\:\sigma_{\hat n_L+j}=\vartheta_j,\:\:N_1<j\leq N_2\right\}\right)=\nonumber\\
&&=
\mathrm Q^{(0)}\big(\,(a,b)\times(c,d)\times\{\underline\vartheta\}\,\big)\label{statement-rtltR}
\end{eqnarray}
\end{theorem}
Theorem \ref{renewal-type-limit-theorem-for-R} is more general than the one given in \cite{Cellarosi} (Theorem 1.6 therein) because it also includes the $R$-denominator $\hat q_{\hat n_L-1}$ preceding the renewal time $\hat n_L$. However, it is a special case of Theorem \ref{theorem:joint-limit-distr} (whose proof is sketched in Appendix \ref{AppendixA}). 
Let us just mention that it relies on the mixing property of a suitably defined special flow over the natural extension $\hat R$ of $R$. The same strategy was used before by Sinai and Ulcigrai \cite{Sinai-Ulcigrai08} in the proof of the analogous statement for Euclidean continued fractions. Another remarkable result in this direction is due to Ustinov \cite{Ustinov-On-statistical} who provides an explicit expression and an approximation, with an error term of order $\mathcal{O}\!\left(\frac{\log L}{L}\right)$, for their limiting distribution function.

\subsection{Estimates of the growth of $\Sigma$-entries}
In this section we prove a number of estimates for the growth of $\Sigma$-entries. The analogous results for Euclidean continued fraction expansions are well known, but in our case the proofs are more involved.

Recall that $\alpha=( h_1\cdot m_1^{\zeta_1},h_2\cdot m_2^{\zeta_2},\ldots)\in\Sigma^\N$. Let us fix a sequence $\underline{\sigma}=\{\sigma_j\}_{j\in\N}\in\Sigma^{\N}$. For every $n$ and every 
$s\cdot t^\zeta\in\Sigma$, set 
\begin{eqnarray}
&&J_{n}=J_n(\underline{\sigma}):=\{\alpha:\:h_j\cdot m_j^{\zeta_j}=\sigma_j,\:j=1,\ldots,n\},\hspace{.3cm}\mbox{and}\nonumber\\
&&J_{n+1}[s\cdot t^\zeta]=J_{n+1}(\underline{\sigma})[s\cdot t^\zeta]:=\{\alpha\in J_n:\:h_{n+1}\cdot m_{n+1}^{\zeta_{n+1}}=s\cdot t^\zeta\}\subset J_n.\nonumber
\end{eqnarray}
\begin{lem}\label{lemma lemmaconstants}
Let $J_{n}$ and $J_{n+1}[s\cdot t^\zeta]$ be as above. Then 
\begin{equation}\label{lemmaconstants}
\frac{1}{30\,(s+1)^2t^2}\leq\frac{|J_{n+1}[s\cdot t^\zeta]|}{|J_{n}|}\leq\frac{6}{(s+1)^2t^2}
\end{equation}
\end{lem}
\begin{proof}
This proof follows closely the one given by Khinchin concerning Euclidean continued fraction (see \cite{Khinchin35}, \S 12).
Let us introduce the convergents $p_j/q_j$, $j=1,\ldots,\nu_n-1$ associated to $(\sigma_1,\ldots,\sigma_n)$. 
The endpoints of the interval $J_n$ can be written as $$\frac{p_{\nu_n-1}}{q_{\nu_n-1}}\hspace{.5cm}\mbox{and}\hspace{.5cm}\frac{p_{\nu_n-1}-\zeta_n\, p_{\nu_n-2}}{q_{\nu_n-1}-\zeta_n\, q_{\nu_n-2}}.$$
Applying the recurrent relations (\ref{recurrent-relations-p-q}) $s+1$ times we define the convergents $p_j/q_j$, $j=1,\ldots,\nu_n+s=\nu_{n+1}-1$ corresponding to $(\sigma_1,\ldots,\sigma_n,s\cdot t^\zeta)$. The endpoints of the interval $J_{n+1}[s\cdot t^\zeta]$ are $$\frac{p_{\nu_{n+1}-1}}{q_{\nu_{n+1}-1}}\hspace{.5cm}\mbox{and}\hspace{.5cm}\frac{p_{\nu_{n+1}-1}-\zeta\, p_{\nu_{n+1}-2}}{q_{\nu_{n+1}-1}-\zeta\, q_{\nu_{n+1}-2}},$$
where $q_{\nu_{n+1}-2}=(s+1)q_{\nu_n-1}+s\,\zeta_n\,q_{\nu_n-2}$ and $q_{\nu_{n+1}-1}=(2t(s+1)-s)q_{\nu_n-1}+(2ts-s+1)\zeta_n\,q_{\nu_n-2}$ (the values of the corresponding numerators are unimportant).
Using the formula (\ref{formula pnqnp1-pnp1qn}) and setting $x=\frac{q_{\nu_n-2}}{q_{\nu_n-1}}$ we obtain
\begin{eqnarray}
&&\hspace{-1.cm}\frac{|J_{n+1}[s\cdot t^\zeta]|}{|J_{n}|}=\frac{q_{\nu_n-1}\left(q_{\nu_n-1}+\zeta_n\,q_{\nu_n-2}\right)}{q_{\nu_{n+1}-1}\left(q_{\nu_{n+1}-1}+\zeta\,q_{\nu_{n+1}-2}\right)}=\nonumber\\
&&\hspace{-1.cm}=\frac{1}{(s+1)^2t^2}\frac{(1+\zeta_n\,x)}{\left(2-\frac{s}{(s+1)t}+\zeta_n\,x\frac{2st-s+1}{(s+1)t}\right)\hspace{-.12cm}\left(2-\frac{s}{(s+1)t}+\zeta_n\,x\frac{2st-s+1+\zeta s}{(s+1)t}+\frac{\zeta}{t}\right)}=\nonumber\\
&&\hspace{-1.cm}=\frac{1}{(s+1)^2t^2}\frac{\mathrm{A}}{\mathrm{B}\,\mathrm{C}},\label{lemmaconstants1}
\end{eqnarray}
where $\mathrm{A}$, $\mathrm{B}$ and $\mathrm{C}$ correspond to the terms in parentheses. We distinguish two main cases: (i) $\zeta_n=+1$ and (ii) $\zeta_n=-1$.
\begin{itemize}
\item[(i)] If $\zeta_n=+1$, then $0\leq x\leq 1$ and we get 
\begin{eqnarray}
1\leq \mathrm{A}\leq 2,\hspace{1cm}1\leq \mathrm{B}\leq 4,\hspace{1cm}1\leq \mathrm{C}\leq 5.\label{lemmaconstants-case-i}
\end{eqnarray}
The above estimates for $\mathrm{A}$ and $\mathrm{B}$ are elementary; the one for $\mathrm{C}$ is obtained discussing the cases $\zeta=+1$ ($\Rightarrow t\geq1$) and $\zeta=-1$ ($\Rightarrow t\geq2$) separately and is also elementary.
\item[(ii)] If $\zeta_n=-1$, then $m_n\geq2$ and by (\ref{recurrent-relations-p-q}) $0\leq x\leq\frac{1}{3}$. We get
\begin{eqnarray}
\frac{2}{3}\leq \mathrm{A}\leq 1,\hspace{1cm}\frac{2}{3}\leq \mathrm{B}\leq 2,\hspace{1cm}\frac{1}{2}\leq \mathrm{C}\leq 3.\label{lemmaconstants-case-ii}
\end{eqnarray}
\end{itemize}
Now, (\ref{lemmaconstants1}), (\ref{lemmaconstants-case-i}) and (\ref{lemmaconstants-case-ii}) give
\begin{equation}\nonumber
\frac{1}{30\,(s+1)^2t^2}=\frac{1}{(s+1)^2t^2}\frac{\frac{2}{3}}{4\cdot
5}\leq\frac{|J_{n+1}[s\cdot t^\zeta]|}{|J_{n}|}\leq\frac{1}{(s+1)^2t^2}\frac{2}{\frac{2}{3}\cdot
\frac{1}{2}}=\frac{6}{(s+1)^2t^2}.
\end{equation}
\end{proof}
%
The next Lemma estimates the Lebesgue measure of the set of $\alpha$ for which the $\Sigma$-entries $h_j\cdot m_j^{\zeta_j}$ satisfy the inequalities $h_j\leq H_j-1$, $j=1,\ldots,n$, where $\{H_j\}_{j=1}^n$ is an arbitrary sequence.
\begin{lem}\label{lem:lemmaY(H)}
Let $\underline H=(H_1,\ldots,H_n)\in\N^n$ and set $
Y(\underline H)
:=\{\alpha:\:h_1+1<H_1,\ldots,h_n+1<H_n\}$.
Then 
\begin{equation}\label{lemmaY(H)}
\left|Y(\underline H)\right|\geq
\left(1-\frac{1}{H_1}\right)\prod_{j=2}^n\lambda_{H_j},
\end{equation}
where $\lambda_H=:1-\frac{4\pi^2}{H}$.
\end{lem}
\begin{proof}
For $\underline\sigma\in\Sigma^n$ and $\underline H
\in \N^n$, let us define the set 
$$W^{(\underline{\sigma},\underline{H})}_{j,n}:=\left\{\alpha:\:h_i\cdot m_i^{\zeta_i}=\sigma_i,\,i=1,\ldots,j,\:\:h_l<H_l-1,
 l=j+1,\ldots,n\right\}.$$
Notice that 
$W^{(\underline{\sigma},\underline{H}
)}_{n,n}=J_n(\sigma_1,\ldots,\sigma_n)$ and does not depend on $\underline H$
. Moreover, $Y(H_1,\ldots H_n)=W^{(\underline{\sigma},\underline{H}
)}_{0,n}$. 
Consider the following estimate obtained from the second inequality of (\ref{lemmaconstants}): for $S\in\N$
\begin{eqnarray}
\sum_{\tiny{\begin{array}{c}s \geq S-1\\t^\zeta\in\Omega^*\end{array}}
}\left|J_{n+1}[s\cdot t^\zeta]\right|
\leq12\,|J_n|\hspace{-.1cm}\sum_{\tiny{\begin{array}{c}s\geq S-1\\ t\in\N\end{array}} 
}\frac{1}{(s+1)^2t^2}\leq\frac{4\pi^2|J_n|}{S
}.\label{estimate1}
\end{eqnarray}
Now (\ref{estimate1}) yields 
\begin{eqnarray}
\left|W^{(\underline{\sigma},\underline{H}
)}_{n-1,n}\right|&=&
\sum_{\tiny{\begin{array}{c}h_n<H_n-1\\m_n^{\zeta_n}\in\Omega^*
\end{array}}}\left|J_{n}(\sigma_1,\ldots,\sigma_{n-1},h_n\cdot m_n^{\zeta_n})\right|=
|J_{n-1}|-\sum_{\tiny{\begin{array}{c}
h_n \geq H_n-1\\
m_n^{\zeta_n}\in\Omega^*
\end{array}}}\left|J_{n}[h_n\cdot m_n^{\zeta_n}]\right|\geq
\nonumber\\
&\geq&|J_{n-1}|\left(1-\frac{4\pi^2}{H_n
}\right)=\lambda_{H_n
}\left|W^{(\underline{\sigma},\underline{H}
)}_{n-1,n-1}\right|,\label{estimate2}
\end{eqnarray}

where $\lambda_{H_n
}=\left(1-\frac{4\pi^2}{H_n
}\right)$. Considering the sum for $h_{n-1}<H_{n-1}-1,\:m_{n-1}^{\zeta_{n-1}}\in\Omega^*
$ in (\ref{estimate2}) we get
\begin{eqnarray}
&&\left|W^{(\underline{\sigma},\underline{H}
)}_{n-2,n}\right|\geq\lambda_{H_n
}\left|W^{(\underline{\sigma},\underline{H}
)}_{n-2,n-1}\right|\geq\lambda_{H_n
}\cdot\lambda_{H_{n-1}
}\left|W^{(\underline{\sigma},\underline{H}
)}_{n-2,n-2}\right|.\label{estimate3}
\end{eqnarray}
Iterating (\ref{estimate3}) we come to
\begin{eqnarray}
&&\left|W^{(\underline{\sigma},\underline{H}
)}_{1,n}\right|\geq
\prod_{j=2}^n\lambda_{H_j
}
\left|W^{(\underline{\sigma},\underline{H}
)}_{1,1}\right|=\prod_{j=2}^n\lambda_{H_j
}
\left|J_1(h_1\cdot m_1^{\zeta_1})\right|.\nonumber
\end{eqnarray}
and summing over $h_{1}<H_{1}-1,\:m_{1}^{\zeta_{1}}\in\Omega^*$ we get the desired estimate (\ref{lemmaY(H)}):
\begin{eqnarray}
|Y(\underline H)|=\left|W^{(\underline{\sigma},\underline{H}
)}_{0,n}\right|\geq
\left(1-\frac{1}{H_1}\right)\prod_{j=2}^n\lambda_{H_j
}
.\nonumber
\end{eqnarray}
\end{proof}
Now we provide an estimate which will be useful later. Let us fix a sequence $\underline{\sigma}=\{\sigma_j\}_{j\in\N}\in\Sigma^\N$ and
let $J_n$ and $J_{n+1}[s\cdot t^\xi]$ be as before. Moreover, set 
\begin{eqnarray}
&&J_{n-1}'=J_{n-1}(\underline{\sigma}):=\{\alpha:\:h_j\cdot m_j^{\zeta_j}=\sigma_j,\:j=2,\ldots,n\}\hspace{.3cm}\mbox{and}\nonumber\\
&&J_{n}'[s\cdot t^\zeta]=J_{n}'(\underline{\sigma})[s\cdot t^\zeta]:=\{\alpha\in J_n':\:h_{n+1}\cdot m_{n+1}^{\zeta_{n+1}}=s\cdot t^\zeta\}\subset J_{n-1}'.\nonumber
\end{eqnarray}
\begin{lem}\label{lem: two-ratios}
\begin{equation}\nonumber
\left|\frac{|J_{n+1}[s\cdot t^\zeta]|}{|J_n|}\cdot\frac{|J_{n-1}'|}{|J_n'[s\cdot t^\zeta]|}-1\right|\leq C_8e^{-C_9 n}
\end{equation}
for some constants $C_8,C_9>0$.
\end{lem}
\begin{proof}
Let us observe that $R J_{n-1}'(\underline{\sigma})=J_{n-1}(\underline{\sigma}')$ and $R J_n'(\underline{\sigma})[s\cdot t^\zeta]=J_n(\underline{\sigma}')[s\cdot t^\zeta]$,
where $\underline{\sigma}'=\{\sigma'_j\}_{j\in\N}$ and $\sigma'_j=\sigma_{j+1}$. We have
\begin{equation}\nonumber
|J_{n-1}'|=\int_{J_{n-1}'}\!\mathbf{1}\,\de x=\int_{J_{n-1}(\underline{\sigma}')}\!\!\!\mathcal{P}(\mathbf{1})(x)\,\de x,\hspace{.5cm}|J_{n}'[s\cdot t^\zeta]|=\int_{J_{n}'[s\cdot t^\zeta]}\!\!\mathbf{1}\,\de x=\int_{J_n(\underline{\sigma}')[s\cdot t^\zeta]}\!\!\!\!\mathcal{P}(\mathbf{1})(x)\,\de x,
\end{equation}
where $\mathcal{P}$ is the Perron-Frobenius operator associated to $R$. The density $\mathcal{P}(\mathbf{1})(x)$ is computed as follows. The cylinders of rank one are of the form 
\begin{eqnarray}
J_1(h\cdot m^+)=\left(\frac{1+2 m h}{1+2m(h+1)},1+\frac{1-2m}{2m(h+1)-h}\right],\nonumber\\
J_1(h\cdot m^-)=\left(1+\frac{1-2m}{2m(h+1)-h},\frac{1+2h(m-1)}{2m(h+1)-2h-1}\right],\nonumber
\end{eqnarray}
and $R|_{J_1(h\cdot m^\pm)}(x)=\mp2m\pm\frac{1+h(x-1)}{h(x-1)+x}$. Therefore $\left(R|_{J_1(h\cdot m^\pm)}\right)'(y)=\mp(h-(h+1)y)^{-2}$ and $(R|_{J_1(h\cdot m^\pm)})^{-1}(x)=\frac{2hm-h+1\pm hx}{2hm+2m-h\pm(h+1)x}=:y_{h\cdot m^\pm}$. We get
\begin{eqnarray}
&&\mathcal{P}(\mathbf{1})(x)=\sum_{y\in R^{-1}(x)}|R'(y)|^{-1}=\sum_{h\cdot m^\zeta\in\Sigma}(h-(h+1)y_{h\cdot m^\zeta})^{2}=\nonumber\\
&&=\sum_{h\geq0}\left(\sum_{m\geq1}\frac{1}{(2hm+2m-h+(h+1)x)^2}+\sum_{m\geq2}\frac{1}{(2hm+2m-h-(h+1)x)^2}\right)=\nonumber\\
&&=\sum_{h\geq0}\frac{1}{4(h+1)^2}\left(\psi^{(1)}\left(\frac{h+2+(h+1)x}{2h+2}\right)+\psi^{(1)}\left(\frac{3h+4-(h+1)x}{2h+2}\right)\right),\nonumber
\end{eqnarray}
where $\psi^{(1)}(x)=\frac{\de}{\de x}\frac{\Gamma'(x)}{\Gamma(x)}$ is the derivative of the digamma function. Notice that the function $\mathcal{P}(\mathbf{1})$ is differentiable and strictly decreasing on $[0,1]$; moreover 
\begin{equation}\label{eq: two-ratios0}
\mathcal{P}(\mathbf{1})'(0)\simeq-0.88575>-1\hspace{.5cm}\mbox{and}\hspace{.5cm}\mathcal{P}(\mathbf{1})'(1)=0.
\end{equation}
By the mean value theorem
\begin{equation}\label{eq: two-ratios1}
|J_{n-1}'|=\mathcal{P}(\mathbf{1})(x_1)\cdot|J_{n-1}(\underline{\sigma}')|\hspace{.5cm}\mbox{and}\hspace{.5cm}|J_{n}'[s\cdot t^\zeta]|=\mathcal{P}(\mathbf{1})(x_2)\cdot |J_n(\underline{\sigma}')[s\cdot t^\zeta]|,
\end{equation}
for some $x_1\in J_{n-1}(\underline{\sigma}')$ and $x_2\in J_n(\underline{\sigma}')[s\cdot t^\zeta]$. 

Let $\{p_j/q_j\}_{j\in\N}$ and $\{p_j'/q_j'\}_{j\in\N}$ be the sequences of $T$-convergents corresponding to $\underline{\sigma}$ and $\underline{\sigma}'$ respectively.
Set $x=\frac{q_{\nu_n-2}}{q_{\nu_n-1}}$ and $x'=\frac{q_{\nu_{n-1}-2}'}{q_{\nu_{n-1}-1}'}$. The ECF-expansions of $x$ and $x'$ coincide up to the $(n-1)$-st $R$-digit (see \cite{Cellarosi}, 
Lemma A.1) and therefore, by Lemma \ref{growth of R-denominators}(i), we have $|x-x'|\leq 3^{\frac{1-n}{3}}$.
Now, by (\ref{eq: two-ratios1}) and (\ref{lemmaconstants1}),  
we get
\begin{eqnarray}
&&\hspace{-1.5cm}
\frac{|J_{n+1}[s\cdot t^\zeta]|}{|J_n|}\cdot\frac{|J_{n-1}'|}{|J_n'[s\cdot t^\zeta]|}
=\nonumber\\
&&\hspace{-1.5cm}=\frac{(1+\zeta_n\,x)\left(2-\frac{s}{(s+1)t}+\zeta_{n}\,x'\frac{2st-s+1}{(s+1)t}\right)\left(2-\frac{s}{(s+1)t}+\zeta_{n}\,x'\frac{2st-s+1+\zeta s}{(s+1)t}+\frac{\zeta}{t}\right)\mathcal{P}(\mathbf{1})(x_1)}{(1+\zeta_{n}\,x')\left(2-\frac{s}{(s+1)t}+\zeta_n\,x\frac{2st-s+1}{(s+1)t}\right)\left(2-\frac{s}{(s+1)t}+\zeta_n\,x\frac{2st-s+1+\zeta s}{(s+1)t}+\frac{\zeta}{t}\right)\mathcal{P}(\mathbf{1})(x_2)}
.\label{eq: two-ratios2}
\end{eqnarray}
Noticing that $\zeta_n x\geq-\frac{1}{3}$ one can show that
\begin{eqnarray}
&\left|\frac{(1+\zeta_n\,x)}{(1+\zeta_{n}\,x')}-1\right|,
\left|\frac{\left(2-\frac{s}{(s+1)t}+\zeta_{n}\,x'\frac{2st-s+1}{(s+1)t}\right)}{\left(2-\frac{s}{(s+1)t}+\zeta_{n}\,x\frac{2st-s+1}{(s+1)t}\right)}-1\right|,
\left|\frac{\left(2-\frac{s}{(s+1)t}+\zeta_{n}\,x'\frac{2st-s+1+\zeta s}{(s+1)t}+\frac{\zeta}{t}\right)}{\left(2-\frac{s}{(s+1)t}+\zeta_n\,x\frac{2st-s+1+\zeta s}{(s+1)t}+\frac{\zeta}{t}\right)}-1\right|
\leq\frac{3^{\frac{4-n}{3}}}{2}.\nonumber
\end{eqnarray}
Let us now consider the term $\frac{\mathcal{P}(\mathbf{1})(x_1)}{\mathcal{P}(\mathbf{1})(x_2)}$. To get estimates of it from above and below we can replace $x_1$ and $x_1$ with appropriate endpoints of  
$J_{n-1}(\underline{\sigma}')$ and $J_n(\underline{\sigma}')[s\cdot t^\zeta]$.
Since $J_n(\underline{\sigma}')[s\cdot t^\zeta]\subset J_{n-1}(\underline{\sigma}')$, those four endpoints can be ordered in four different ways. Let us discuss only one of those cases, the others being similar. 

Let the endpoints $y_1=\frac{p_{\nu_{n-1}-1}'}{q_{\nu_{n-1}-1}'}$, $y_2=\frac{p_{\nu_{n-1}-1}'-\zeta_{n}\, p_{\nu_{n-1}-2}'}{q_{\nu_{n-1}-1}'-\zeta_{n}\, q_{\nu_{n-1}-2}'}$, $z_1=\frac{p_{\nu_{n}-1}'}{q_{\nu_{n}-1}'}$, $z_2=\frac{p_{\nu_{n}-1}'-\zeta\, p_{\nu_{n}-2}'}{q_{\nu_{n}-1}'-\zeta\, q_{\nu_{n}-2}'}$ be arranged as follows:
$0<y_1<z_1< z_2<y_2<1$.
Then, since the function $\mathcal{P}(\mathbf{1})$ is decreasing, $y_1\leq x_1\leq y_2$ and $z_1\leq x_2\leq z_2$, we get $$\frac{\mathcal{P}(\mathbf{1})(x_1)}{\mathcal{P}(\mathbf{1})(x_2)}\leq
1+\frac{\mathcal{P}(\mathbf{1})(z_2)-\mathcal{P}(\mathbf{1})(y_1)}{\mathcal{P}(\mathbf{1})(y_1)}$$
Let us use (\ref{eq: two-ratios0}), the fact that $z_2$ and $y_1$ have the same $R$-expansion up to the $(n-1)$-st digit, (\ref{lemmaconstants}) and the fact that $\mathcal{P}(\mathbf{1})(1)\simeq0.90238$:
\begin{eqnarray}
\frac{\left|\mathcal{P}(\mathbf{1})(z_2)-\mathcal{P}(\mathbf{1})(y_1)\right|}{\mathcal{P}(\mathbf{1})(y_1)}\leq\frac{|z_2-y_1|}{\mathcal{P}(\mathbf{1})(y_1)}\leq\frac{C_{10}\,3^{\frac{1-n}{3}}}{(s+1)t\,\mathcal{P}(\mathbf{1})(y_1)}\leq C_{11}\,3^{\frac{1-n}{3}}\nonumber
\end{eqnarray}
for some constants $C_{10},C_{11}>0$. 
Thus we get the desired estimate:
\begin{eqnarray}
\left|\frac{|J_{n+1}[s\cdot t^\zeta]|}{|J_n|}\cdot\frac{|J_{n-1}'|}{|J_n'[s\cdot t^\zeta]|}-1\right|\leq 
C_{8}e^{-C_9 n},\nonumber
\end{eqnarray}
for some $C_8,C_9>0$.
\end{proof}
\section{Iterated Renormalization of $\gamma_{\alpha,N}$
}\label{iterated-renormalization-of-S}
In Section \ref{AERFs}  we discussed the renormalization of $\gamma_{\alpha, N}$, i.e. the procedure which \virg{erases} the geometric structure at smallest scale in the curve $\gamma_{\alpha,N}$.
Now we want to iterate the renormalization formula (\ref{erf1}). In order to do this, we consider $\alpha_0:=\alpha$, $\alpha_l:=T^l(\alpha_0)$ (as in Section \ref{ECF}), $N_0:=N$, $N_{l}:=\alpha_{l-1}\,N_{l-1}$, $\eta_0:=1$ and $\eta_{l}:=\eta(\alpha_{l-1})$, $l\in\N$. Define also $\kappa_0:=0$, $\kappa_l:=\kappa_l(\alpha):=1+\eta_1+\eta_1\eta_2+\cdots+\eta_1\eta_2\cdots\eta_{l-1}$. 
With these notations, iterating (\ref{erf1}) $r$ times we get
\begin{eqnarray}
\su{\alpha}{t\,N}&=&(\alpha_0\cdots\alpha_{r-1})^{-\ha}\,\Big(\exp\left\{\kappa_r\,\pifi\right\}\,
\s{\alpha_r}{\eta_1\cdots\eta_r}{t\,N_r}+\nonumber\\
&&+\exp\left\{\kappa_{r-1}\,\pifi\right\}\alpha_{r-1}^\ha\Gamma^{(\eta_1\cdots\eta_{r-1})}(\alpha_{r-1},t \,N_{r-1})+\nonumber\\
&&+\exp\left\{\kappa_{r-2}\,\pifi\right\}\left(\alpha_{r-2}\alpha_{r-1}\right)^\ha\Gamma^{(\eta_1\cdots\eta_{r-2})}(\alpha_{r-2},t\,N_{r-2})+\ldots+\nonumber\\
&&+\exp\left\{\kappa_{r-j}\,\pifi\right\}\left(\alpha_{r-j}\cdots\alpha_{r-1}\right)^\ha\Gamma^{(\eta_1\cdots\eta_{r-j})}(\alpha_{r-j},t\,N_{r-j})+\ldots+\nonumber\\
&&+\exp\left\{\kappa_{0}\,\pifi\right\}\left(\alpha_{0}\cdots\alpha_{r-1}\right)^\ha\Gamma^{(1)}(\alpha_{0},t\,N_{0})\Big)=\nonumber\\
&=&(\alpha_0\cdots\alpha_{r-1})^{-\ha}\Big(\exp\left\{\kappa_r\,\pifi\right\}\s{\alpha_r}{\eta_1\cdots\eta_r}{t\,N_r}+\nonumber\\
&&+\sum_{j=1}^r\exp\left\{\kappa_{r-j}\,\pifi\right\}(\alpha_{r-j}\cdots\alpha_{r-1})^\ha\,\Gamma^{(\eta_1\cdots\eta_{r-j})}(\alpha_{r-j},t\,N_{r-j})\Big).
\label{iteratederf1}
\end{eqnarray}

Our next step is to choose $r$ as a function of $N$ and $\alpha$ so that $N_r=\alpha_0\cdots\alpha_{r-1}\,N$ is $\mathcal{O}(1)$, that is $(\alpha_0\cdots\alpha_{r-1})^{-\ha}=\mathcal{O}(\sqrt{N})$. We make use of the relation (\ref{prodalpha}) and we define $r$ in terms of the $R$-denominator corresponding to the renewal time $\hat n_N$. For $\alpha=(h_1\cdot m_1^\pm,h_2\cdot m_2^\pm,\ldots)\in\Sigma^\N$, 
set $r=r(\alpha,N):=\nu_{\hat n_N}-1=h_1+\ldots+h_{\hat n_N}+\hat n_N$, where $\hat n_N=\min\{n\in\N:\:\hat q_n>N\}$ as in Theorem \ref{renewal-type-limit-theorem-for-R}. Define
$\alpha_0\cdots\alpha_{r(\alpha,N)-1}\,N=N_{r(\alpha,N)}=:\Theta_\alpha(N)$. We have the following
\begin{prop}\label{prop: limiting distribution for Theta}
$\Theta_\alpha(N)$ has a limiting probability distribution on $(0,\infty)$ w.r.t. $\mu_R$ as $N\rightarrow\infty$. In other words: there exists a probability measure $\mathrm Q^{(1)}$ on $(0,\infty)$ such that for every $0<a< b$ we have
\begin{equation}\nonumber
\lim_{N\rightarrow\infty}\mu_R\left(\left\{\alpha:\:\:a<\Theta_\alpha(N)< b\right\}\right)=\mathrm Q^{(1)}\big((a,b)\big).
\end{equation}
\end{prop}
\begin{proof}
Our goal is to write $\Theta_\alpha(N)$ as a function of $\hat q_{\hat n_N-1}/N$, $\hat q_{\hat n_N}/N$ and a finite number of $\Sigma$-entries of $\alpha$ preceding and/or following the renewal time $\hat n_N$.
By (\ref{prodalpha}) we have
\begin{eqnarray}
\Theta_\alpha(N)=\alpha_0\cdots\alpha_{\nu_{\hat n_N}-2}\,N=\left(\frac{q_{\nu_{\hat n_N}-1}}{N}+\xi_{\nu_{\hat n_N}-1}\cdot\alpha_{\nu_{\hat n_N}-1}\cdot\frac{q_{\nu_{\hat n_N}-2}}{N}\right)^{-1}.\label{proof: limiting distribution for Theta 1}
\end{eqnarray}
In order to write $q_{\nu_{\hat n_N}-1}$ and $q_{\nu_{\hat n_N}-2}$ in terms of $\hat q_{\hat n_N}=q_{\nu_{\hat n_N}}$ and $\hat q_{\hat n_N-1}=q_{\nu_{\hat n_N-1}}$ we use the recurrent relation (\ref{recurrent-relations-p-q}) for the ECF-de\-no\-mi\-na\-tors, getting the $h_{\hat n}\times h_{\hat n}$ linear system
\begin{eqnarray}
\begin{bmatrix}
2\,k_{\nu_{\hat n}}&\xi_{\nu_{\hat n}-1}&&&&\\
-1&2\,k_{\nu_{\hat n}-1}&-1&&&\\
&-1&2&\ddots&&\\
&&-1&\ddots&-1&\\
&&&\ddots&2&-1\\
&&&&-1&2
\end{bmatrix}
\cdot
\begin{bmatrix}
q_{\nu_{\hat n}-1}\\
q_{\nu_{\hat n}-2}\vspace{-.06cm}\\
\vdots\vspace{-.08cm}\\
q_{\nu_{\hat n}-j}\vspace{-.08cm}\\
\vdots\vspace{-.1cm}\\
q_{\nu_{\hat n}-(h_{\hat n}-1)}\\
q_{\nu_{\hat n}-h_{\hat n}}
\end{bmatrix}
=\begin{bmatrix}
\hat q_{\hat n}\vspace{.06cm}\\
0\vspace{-.11cm}\\
\vdots\vspace{0cm}\\
0\vspace{-.11cm}\\
\vdots\\
0\\
\hat q_{\hat n-1}
\end{bmatrix}
\label{linear system}
\end{eqnarray}
where $\hat n=\hat n_N$. The quantities $k_{\nu_{\hat n}-1}^{\xi_{\nu_{\hat n}-1}}=m_{\hat n}^{\zeta_{\hat n}}\in\Omega^*$ and $k_{\nu_{\hat n}}\in\N$, along with the size $h_{\hat n}$ of the linear system,  depend only on the two $\Sigma$-entries $(h_{\hat n}\cdot m_{\hat n}^{\zeta_{\hat n}},h_{\hat n+1}\cdot m_{\hat n+1}^{\zeta_{\hat n+1}})\in\Sigma^2$. We are interested in the first two entries of the solution of (\ref{linear system}). One can check that
\begin{eqnarray}
&&q_{\nu_{\hat n}-1}=\frac{\left((2h_{\hat n}-2)k_{\nu_{\hat n}-1}-(h_{\hat n}-2)\right)\hat q_{\hat n}-\xi_{\nu_{\hat n}-1}\hat q_{\hat n-1}}{(4h_{\hat n}-4)k_{\nu_{\hat n}}k_{\nu_{\hat n}-1}-(2h_{\hat n}-4)k_{\nu_{\hat n}}+(n-1)\xi_{\nu_{\hat n}-1}}\hspace{.4cm}\mbox{and}\nonumber\\
&&q_{\nu_{\hat n}-2}=\frac{\left(h_{\hat n}-1\right)\hat q_{\hat n}+2k_{\nu_{\hat n}}\hat q_{\hat n-1}}{(4h_{\hat n}-4)k_{\nu_{\hat n}}k_{\nu_{\hat n}-1}-(2h_{\hat n}-4)k_{\nu_{\hat n}}+(n-1)\xi_{\nu_{\hat n}-1}}.\label{proof: limiting distribution for Theta 2}
\end{eqnarray}
Therefore (\ref{proof: limiting distribution for Theta 1}) and (\ref{proof: limiting distribution for Theta 2}) show that $\Theta_\alpha(N)$ is a function of $\hat q_{\hat n_N-1}/N\in(0,1]$, $\hat q_{\hat n_N}/N\in(1,\infty)$, $(h_{\hat n_N}\cdot m_{\hat n_N}^{\zeta_{\hat n_N}},h_{\hat n_N+1}\cdot m_{\hat n_N+1}^{\zeta_{\hat n_N+1}})\in\Sigma^2$ and $\alpha_{\nu_{\hat n_N}-1}=R^{\hat n_N}(\alpha)$. 
Now, by Theorem \ref{renewal-type-limit-theorem-for-R}, we obtain the existence of  a limiting distribution as $N\rightarrow\infty$, w.r.t. $\mu_R$.
\end{proof}
\subsection{Approximation of $\gamma_{\alpha,N}$
}
In this section we construct an approximation for the curve $\gamma_{\alpha,N}$. The approximating curve $\gamma_{\alpha,N}^J$ will contain only the $J$ largest geometric scales (corresponding to $J$ iterations of the jump transformation $R$).
Having specified our choice for $r$, we can also regroup the $\nu_{\hat n_N}
$ 
terms in (\ref{iteratederf1}) involving $\Gamma$'s into $\hat n_N$ 
terms as follows:
\begin{eqnarray}
\Delta_l(t)=\Delta_l(t;\alpha,N):=\sum_{j=2}^{h_l+2}\exp\left\{\kappa_{\nu_l-j\,\pifi}\right\}\left((\alpha)_{\nu_l-j}^{\nu_l-2}\right)^\ha\,\Gamma^{(\eta_1\cdots\eta_{\nu_l-j})}(\alpha_{\nu_l-j},t\,N_{\nu_l-j})
\label{hatDelta}
\end{eqnarray}
for $l=1,\ldots,\hat n_N$, where $(\alpha)_{l_1}^{l_2}:=\alpha_{l_1}\cdots\alpha_{l_2}$ if $l_1\leq l_2$ and $(\alpha)_{l_1}^{l_2}:=1$ if $l_1>l_2$. Also recall that $\nu_{l-1}=\nu_{l}-h_l-1$. Formula (\ref{iteratederf1}) becomes now
\begin{eqnarray}
\gamma_{\alpha,N}(t)=\frac{\su{\alpha}{t\,N}}{\sqrt N}&=&\Theta_\alpha^{-\ha}(N)\,\Big(\exp\left\{\kappa_{\nu_{\hat n}-1}\,\pifi\right\}\,
\s{\alpha_{\nu_{\hat n}-1}}{\eta_1\cdots\eta_{\nu_{\hat n}-1}}{t\,\Theta_\alpha(N)}+\nonumber\\
&&+\sum_{j=0}^{\hat n-1}\left((\alpha)_{\nu_{\hat n-j}-1}^{\nu_{\hat n}-2}\right)^\ha\Delta_{\hat n-j}(t)\Big),
\label{iteratederf2}
\end{eqnarray}
where $\hat n=\hat n_N$.  The following Lemma proves that, on a set of arbitrarily large measure, the product $\left((\alpha)_{\nu_{\hat n-j}-1}^{\nu_{\hat n}-2}\right)^\ha\Delta_{\hat n-j}(t)$ decays sufficiently fast as $j$ grows. One can assume that $\hat n$ is large enough so that $\hat n-j\geq 1$. This is the case because later we shall let $N\rightarrow\infty$ and hence $\hat n_N\rightarrow\infty$. 
\begin{lem}\label{localizinglemma}
For all sufficiently large $J$ 
\begin{equation}\label{eq:localizing-lemma-1}
\mu_R\left(\left\{\alpha:\:\left|\left((\alpha)_{\nu_{\hat n-j}-1}^{\nu_{\hat n}-2}\right)^\ha\Delta_{\hat n-j}(t)\right|\leq C_{12}\, e^{-C_{13}\, j},\:j=J,\ldots,\hat n-1\right\}\right)\geq 1-\delta_1(J),
\end{equation}
where $C_{12},C_{13}>0
$ are some constants and $\delta_1(J)\rightarrow0$ as $J\rightarrow\infty$.
\end{lem}
\begin{proof}
Notice that, by (\ref{arf2}), for every $l=2,\ldots,h_{\hat n-j}+2$, $$\left|\alpha_{\nu_{\hat n-j}-l}^\ha\,\Gamma(\alpha_{\nu_{\hat n-j}-l},t\,N_{\nu_{\hat n-j}-l})\right|\leq C_5+C_6\,\alpha_{\nu_{\hat n-j}-l}^\ha\leq C_{12},$$ where $C_{12}:=C_5+C_6$. Now, by (\ref{hatDelta}), $$\left|\Delta_{\hat n-j}(t)\right|\leq\sum_{l=2}^{h_{\hat n-j}+2}\left((\alpha)_{\nu_{\hat n-j}-l}^{\nu_{\hat n-j}-2}\right)^\ha\,\Gamma(\alpha_{\nu_{\hat n-j}-l},t\,N_{\nu_{\hat n-j}-l})\leq C_{12}(h_{\hat n-j}+1).$$ 
By construction of the jump transformation $R$, exactly one of the factors in $(\alpha)_{\nu_{\hat n-j}-1}^{\nu_{\hat n-j+1}-2}$ is less then $\frac{1}{2}$. Therefore for every $j=1,\ldots,\hat n-1$
$$\left|\left((\alpha)_{\nu_{\hat n-j}-1}^{\nu_{\hat n}-2}\right)^\ha\Delta_{\hat n-j}(t)\right|\leq C_{12}\, 2^{-\ha j}\,(h_{\hat n-j}+1).$$
Thus it is enough to show that, for all sufficiently large $J\in N$ and $\hat n$,
\begin{equation}\label{eq:localizing-lemma-2}
\left|\left\{\alpha:\:h_{\hat n-j}\leq e^{C_{14}\,j},\:j=J,\ldots,\hat n-1\right\}\right|\geq 1-\delta_2(J),
\end{equation}
where $0<C_{14}<\frac{\log 2}{2}\simeq0.346574$ and $\delta_2(J)\rightarrow0$ as $J\rightarrow\infty$.
By Lemma \ref{lem:lemmaY(H)}, setting $\underline H=(e^{C_{14}\,(\hat n-1)}+1,e^{C_{14}\,(\hat n-2)}+1,\ldots,e^{C_{14}\,J}+1)\in\N^{\hat n-J}$, we get
\begin{eqnarray}
&&\left|\left\{\alpha:\:h_{\hat n-j}\leq e^{C_{14}\,j},\:j=J,\ldots,\hat n-1\right\}\right|=|Y(\underline H)|\geq\nonumber\\
&&\geq\left(1-\frac{1}{e^{C_{14}\,(\hat n-1)}+1}\right)\prod_{j=J}^{\hat n-2}\left(1-\frac{4\pi^2}{e^{C_{14}\,j}+1}\right)\geq\prod_{j=J}^\infty\left(1-4\pi^2e^{-C_{14}\,j}\right)=:\delta_2(J).
\nonumber
\end{eqnarray}
The estimate (\ref{eq:localizing-lemma-2}) is thus proven, along with our initial statement (\ref{eq:localizing-lemma-1}) setting $C_{13}:=\frac{\log2}{2}-C_{14}$.
\end{proof}


For $J\in\N$ define the curve associated to the truncated renormalized sum as
\begin{eqnarray}
t\mapsto\gamma_{\alpha,N}^{J}(t):=\Theta_\alpha^{-\ha}(N)\left(e^{\kappa_{\nu_{\hat n}-1}\pifi}\s{\alpha_{\nu_{\hat n}-1}}{\eta_1\cdots\eta_{\nu_{\hat n}-1}}{t\,\Theta_\alpha(N)}+
\sum_{j=0}^{J-1}\left((\alpha)_{\nu_{\hat n-j}-1}^{\nu_{\hat n}-2}\right)^\ha\Delta_{\hat n-j}(t)\right).\label{eq: gammaJ}
\end{eqnarray} 
The number $J$ corresponds to the number of scales one considers in approximating the curve $\gamma_{\alpha,N}$, starting from the largest scale. The following Lemma shows that $\gamma_{\alpha,N}$ is exponentially well approximated by $\gamma_{\alpha,N}^J$ for a set of $\alpha$'s whose measure tends to 1 as $J$ increases. 
\begin{lem}\label{localizinglemma1}
For all sufficiently large $J$ and $N$
\begin{equation}\label{eq: localizinglemma1}
\mu_R\left(\left\{\left|\gamma_{\alpha,N}(t)-\gamma_{\alpha,N}^{J}(t)\right|\leq e^{-C_{15}J}\right\}\right)\geq 1-\delta_3(J)
\end{equation}
for every $t\in[0,1]$, where $C_{15}>0$ is some constant and $\delta_3(J)\rightarrow0$ as $J\rightarrow\infty$.
\end{lem}
\begin{proof}
Since by Proposition (\ref{prop: limiting distribution for Theta}) $\Theta_\alpha(N)$ has a limiting distribution on $(0,\infty)$ as $N\rightarrow\infty$, so $\Theta_\alpha^{-\ha}(N)$ does. Then, 
for sufficiently large $N$, we have
\begin{equation}\nonumber
\mu_R\left(\left\{\alpha:\:\Theta_\alpha^{-\ha}(N)\leq J\right\}\right)\geq 1-\delta_4(J),
\end{equation}
where $\delta_4(J)\rightarrow0$ as $J\rightarrow\infty$.
On the other hand, by Lemma \ref{localizinglemma}, for sufficiently large $J$ and $N$, 
\begin{eqnarray}
\left|\gamma_{\alpha,N}(t)-\gamma_{\alpha,N}^{J}(t)\right|&=&\Theta_\alpha^{-\ha}(N)\left|\sum_{j={J+1}}^{\hat n-1}\left((\alpha)_{\nu_{\hat n-j}-1}^{\nu_{\hat n}-2}\right)^\ha\Delta_{\hat n-j}(t)\right|\leq C_{12}\,\Theta_\alpha^{-\ha}(N)\sum_{j=J}^{\hat n-1}e^{-C_{13} j}=\nonumber\\
&=&C_{12}\,\Theta_\alpha^{-\ha}(N)\frac{e^{-C_{13} (J-1)}-e^{-C_{13}(\hat n-1)}}{e^{C_{13}}-1}\leq\frac{C_{12}\,e^{C_{13}}}{e^{C_{13}}-1}\,\Theta_\alpha^{-\ha}(N)\,e^{-C_{13} J}\nonumber
\end{eqnarray}
holds for every $t\in[0,1]$ on a set of $\mu_R$-measure bigger than $1-\delta_1(J)$. Therefore 
\begin{eqnarray}
\left|\gamma_{\alpha,N}(t)-\gamma_{\alpha,N}^{J}(t)\right|\leq\frac{C_{12}\,e^{C_{13}}\,J}{e^{C_{13}}-1}\,e^{-C_{13} J}\leq e^{-C_{15}J}
\nonumber
\end{eqnarray}
for some constant $C_{15}>0$ 
on a set of $\mu_R$-measure bigger than $1-\delta_1(J)-\delta_4(J)$. The Lemma is thus proven setting $\delta_3(J):=\delta_1(J)+\delta_4(J)$.
\end{proof}

\subsection{Rewriting of $\gamma_{\alpha,N}^J$ in Terms of Renewal Variables 
}
Now we can study 
the curve $\gamma_{\alpha,N}^J(t)$. Our goal is to rewrite it in terms of $\Theta_\alpha(N)$, $\alpha_{\nu_{\hat n}-1}$ and a finite number of $\Sigma$-entries preceeding the renewal time. We will also need two additional functions, $K_\alpha^8(N)$ and $E_\alpha(N)$ to take into account phase terms and conjugations coming from the renormalization procedure.

For $\alpha=(h_1\cdot m_1^{\zeta_1},h_2\cdot m_2^{\zeta_2},\ldots)\in\Sigma^\N$ we have an explicit expression for $\eta_l$, $l=1,\ldots,\nu_{\hat n}-1$:
\begin{eqnarray}
&&\eta_1=\ldots=\eta_{h_1}=1,\,\eta_{\nu_1-1}=-\zeta_1,\nonumber\\
&&\eta_{\nu_1}=\ldots=\eta_{\nu_1+h_2}=1,\,\eta_{\nu_2-1}=-\zeta_2,\nonumber\\
&&\hspace{3cm}\ldots\nonumber\\
&&\eta_{\nu_{\hat n-1}}=\ldots=\eta_{\nu_{\hat n-1}+h_{\hat n}}=1,\,\eta_{\nu_{\hat n}-1}=-\zeta_{\hat n}.\nonumber
\end{eqnarray}
Thus 
\begin{eqnarray}
\eta_1\cdots\eta_{\nu_l-1}&=&\prod_{s=1}^{l}(-\zeta_s)\hspace{.5cm}\mbox{and}\label{eqprodetas}\\
\kappa_{\nu_l}&=&1+(h_1-\zeta_1)+(-\zeta_1)(h_2-\zeta_2)+(-\zeta_1)(-\zeta_2)(h_3-\zeta_3)+\ldots+\nonumber\\
&&+(-\zeta_1)\cdots(-\zeta_{l-1})(h_l+\zeta_l)=
1+\sum_{j=1}^l(h_j-\zeta_j)\prod_{s=1}^{j-1}(-\zeta_s).\label{eqkappa}
\end{eqnarray}
The following Lemma gives an explicit formula for the partial products along the $T$-orbit of $\alpha$ which appear in (\ref{eq: gammaJ}).
\begin{lem}\label{lemma: partial products}
Let $\alpha=(h_1\cdot m_1^{\zeta_1},h_2\cdot m_2^{\zeta_2},\ldots)\in\Sigma^\N$. Set $\beta_j:=\alpha_{\nu_{\hat n-j}-2}$. Then
\begin{eqnarray}
&&B_{s,j}=B_{s,j}(\alpha):=(\alpha)_{\nu_{\hat n-j}-s}^{\nu_{\hat n-j}-2}=\frac{\beta_j}{(s-1)-(s-2)\beta_j}\label{lemmaprodalpha1},\\
&&D_j=D_j(\alpha):=(\alpha)_{\nu_{\hat n-j}-1}^{\nu_{\hat n}-2}=\prod_{u=0}^{j-1}\frac{\beta_u}{1+h_{\hat n-u}(1-\beta_u)}\label{lemmaprodalpha2}.
\end{eqnarray}
\end{lem}
\begin{proof}
Both identities follow, after telescopic cancellations, from
\begin{equation}
\alpha_{\nu_{\hat n-j}-s}=\frac{(s-2)-(s-3)\beta_j}{(s-1)-(s-2)\beta_j}\label{eqprooflemmaalpha}.
\end{equation}
\end{proof}
Notice that $\beta_j$ is a function of $R^{\hat n_N}(\alpha)$ and $j$ ($j\leq J$) $\Sigma$-entries preceding the renewal time $\hat n_N$.
With the above notation (\ref{eq: gammaJ}) becomes
\begin{eqnarray}
&&\gamma_{\alpha,N}^J(t)
=\Theta_\alpha(N)^{-\ha}\Big(\exp\left\{\kappa_{\nu_{\hat n}-1}\,\pifi\right\}\s{\alpha_{\nu_{\hat n}-1}}{\eta_1\cdots\eta_{\nu_{\hat n}-1}}{t\,\Theta_\alpha(N)}+\nonumber\\
&&+\sum_{j=0}^{J-1}D_j^\ha\sum_{s=2}^{h_{\hat n-j}+2}\exp\left\{\kappa_{\nu_{\hat n-j}-s}\,\pifi\right\}B_{s,j}^\ha\,\Gamma^{(\eta_1\cdots\eta_{\nu_{\hat n-j}-s})}(\alpha_{\nu_{\hat n-j}-s},t\,N_{\nu_{\hat n-j}-s})\Big).\label{gammaJmodified1}
\end{eqnarray}
We want to collect a phase term of the form $\exp\{\kappa_{\nu_{\hat n-J}-1}\pifi\}$ and the corresponding \virg{conjugation} index $(\eta_1\cdots\eta_{\nu_{\hat n-J}-1})$. To do this, using (\ref{eqprodetas}) and (\ref{eqkappa}), we introduce the quantities $\Psi_J$, $\Upsilon_J$, $\mathcal{E}_J$ and $\mathcal{E}_J^j$, depending only on a finite number of $\Sigma$-entries of $\alpha$ preceding the renewal time $\hat n_N$:
\begin{eqnarray}
&&(\kappa_{\nu_{\hat n}-1}-\kappa_{\nu_{\hat n-J}-1})(\eta_1\cdots\eta_{\nu_{\hat n-J}-1})=\nonumber\\
&&=\left(\kappa_{\nu_{\hat n}}-\kappa_{\nu_{\hat n-J}}-\eta_1\cdots\eta_{\nu_{\hat n}-1}+\eta_1\cdots\eta_{\nu_{\hat n-J}-1}\right)(\eta_1\cdots\eta_{\nu_{\hat n-J}-1})=\nonumber\\
&&=\sum_{u=\hat n-J+1}^{\hat n}(h_u-\zeta_u)\prod_{v=\hat n-J+1}^{u-1}(-\zeta_v)-\prod_{v=\hat n-J+1}^{\hat n}(-\zeta_v)+1=:\nonumber\\
&&=:\Psi_J=\Psi_J\left(h_l\cdot m_l^{\zeta_l},\:l=\hat n-J+1,\ldots,\hat n\right),\nonumber
\end{eqnarray}
\begin{eqnarray}
&&(\kappa_{\nu_{\hat n-j}-s}-\kappa_{\nu_{\hat n-J}-1})(\eta_1\cdots\eta_{\nu_{\hat n-J}-1})=\nonumber\\
&&=\left(\kappa_{\nu_{\hat n-j-1}}+(h_{\hat n-j}-s+1)(\eta_1\cdots\eta_{\nu_{\hat n-j-1}-1})-\kappa_{\nu_{\hat n-J}}+(\eta_1\cdots\eta_{\nu_{\hat n-J}-1})\right)(\eta_1\cdots\eta_{\nu_{\hat n-J}-1})=\nonumber\\
&&=\sum_{u=\hat n-J+1}^{\hat n-j-1}(h_u-\zeta_u)\prod_{v=\hat n-J+1}^{u-1}(-\zeta_v)+(h_{\hat n-j}-s+1)\prod_{v=\hat n-J+1}^{\hat n-j-1}(-\zeta_v)+1=:\nonumber\\
&&=:\Upsilon_{s,J}=\Upsilon_{s,J}\left(h_l\cdot m_l^{\zeta_l},\:l=\hat n-J+1,\ldots,\hat n-j\right),\nonumber
\end{eqnarray}
\begin{eqnarray}
&&\mathcal{E}_J:=\eta_{\nu_{\hat n-J}}\cdots\eta_{\nu_{\hat n}-1}=\prod_{v=\hat n-J+1}^{\hat n}(-\zeta_v),\hspace{1cm}\mathcal{E}_{J}^j:=\eta_{\nu_{\hat n-J}}\cdots\eta_{\nu_{\hat n-j}-s}=\prod_{v=\hat n-J+1}^{\hat n-j-1}(-\zeta_v).\nonumber
\end{eqnarray}
Now (\ref{gammaJmodified1}) becomes
\begin{eqnarray}
&&\gamma_{\alpha,N}^J(t)=\exp\left\{\kappa_{\nu_{\hat n-J}-1}\pifi\right\}\Theta_\alpha(N)^{-\ha}\Big(\exp\left\{\Psi_J\pifi\right\}\s{R^{\hat n}(\alpha)}{\mathcal{E}_J}{t\,\Theta_\alpha(N)}+\nonumber\\
&&+\sum_{j=0}^{J-1}D_j^\ha\sum_{s=2}^{h_{\hat n-j}+2}\exp\left\{\Upsilon_{s,J}\pifi\right\}B_{s,j}^\ha\,\Gamma^{(\mathcal{E}_J^j)}(\alpha_{\nu_{\hat n-j}-s},t\,N_{\nu_{\hat n-j}-s})\Big)^{(\eta_1\cdots\eta_{\nu_{\hat n-J}-1})}.\label{gammaJmodified2}
\end{eqnarray}
On the other hand, we also introduce the functions $E_\alpha(N)$ and $K_\alpha(N)$, depending on the entire trajectory of $\alpha$ under the jump transformation $R$ until the renewal time $\hat n_N$  (exactly as $\Theta_{\alpha}(N)$ does):
\begin{eqnarray}
&&E_\alpha(N):=\eta_1\cdots\eta_{\nu_{\hat n}-1}=\prod_{v=1}^{\hat n}(-\zeta_v),\hspace{1cm}K_\alpha(N):=\kappa_{\nu_{\hat n}}=\sum_{u=1}^{\hat n}(h_u-\zeta_u)\prod_{v=1}^{u-1}(-\zeta_v).\nonumber
\end{eqnarray}
Using (\ref{eqprodetas}) and (\ref{lemmaprodalpha1}$\div$\ref{eqprooflemmaalpha}), let us recall that $\alpha_{\nu_{\hat n-j}-s}$ is a function of $\beta_j$ and s; moreover notice that
\begin{eqnarray}
\eta_1\cdots\eta_{\nu_{\hat n-J}-1}&=&\mathcal{E}_J\cdot E_\alpha(N)
\hspace{.4cm}\mbox{and}\nonumber\\
 N_{\nu_{\hat n-j}-s}&=&\alpha_0\cdots\alpha_{\nu_{\hat n-j}-s-1}\cdot N=
 \frac{\Theta_\alpha(N)}{(\alpha)_{\nu_{\hat n-j}-s}^{\nu_{\hat n-j}-2}\cdot(\alpha)_{\nu_{\hat n-j}-1}^{\nu_{\hat n}-2}}=\nonumber 
\frac{\Theta_\alpha(N)}{B_{s,j}\cdot D_j}
\end{eqnarray}
are functions of $\Theta_\alpha(N)$, $E_\alpha(N)$, $R^{\hat n_N}(\alpha)$ and a finite number of $\Sigma$-entries of $\alpha$ preceding the renewal time $\hat n_N$.
Furthermore, by (\ref{proof: limiting distribution for Theta 1}) and (\ref{proof: limiting distribution for Theta 2}), $\Theta_\alpha(N)$ is a function of $\hat q_{\hat n-1}/N$, $\hat q_{\hat n}/N$, $R^{\hat n_N(\alpha)}$ and the two $\Sigma$-entries $(h_{\hat n_N}\cdot m_{\hat n_N}^{\zeta_{\hat n_N}},h_{\hat n_N+1}\cdot m_{\hat n_N+1}^{\zeta_{\hat n_N+1}})$.

In addition to this, since  $\kappa_{\nu_{\hat n-J}-1}$ appears in the phase term of (\ref{gammaJmodified2}) as multiplier of $\pifi$ it is also natural to consider its values modulo 8. Defining
$K_\alpha^{8}(N):=K_\alpha(N)\:\:(\mathrm{mod}\:8)$, we have
$$\kappa_{\nu_{\hat n-J}-1}\equiv K_\alpha^8(N)-E_\alpha(N)\sum_{u=\hat n-J+1}^{\hat n}(h_u-\zeta_u)\,\mathcal{E}_{\hat n-u+1}\:\:\,(\mathrm{mod}\:8)\nonumber.$$
Therefore, we can rewrite (\ref{gammaJmodified2}) as 
\begin{eqnarray}
\gamma_{\alpha,N}^J(t)=\mathrm F_1\!\left(t,R^{\hat n_N}(\alpha),\frac{\hat q_{\hat n_N-1}}{N},\frac{\hat q_{\hat n_N}}{N},K_\alpha^8(N),E_\alpha(N),\left\{h_l\cdot m_l^{\zeta_l},\:\hat n_N-J\leq l\leq \hat n_N\right\}\right),\label{gammaJmodified3}
\end{eqnarray}
where $\mathrm F_1$ is a complex-valued, 
measurable
function of its arguments. Notice that the formul\ae\: (\ref{ERF-Fedotov-Klopp}) and (\ref{erf2}) 
enter into the definition of $\mathrm F_1$, but we shall not use them directly.

Let us recall that Theorem \ref{renewal-type-limit-theorem-for-R} (which is a special case of Theorem \ref{theorem:joint-limit-distr} and generalizes Theorem 1.6 in \cite{Cellarosi}) 
already establishes 
the existence of a limiting probability distribution for $\hat q_{\hat n_N-1}/N$ and $\hat q_{\hat n_N}/N$, jointly with any finite number of $\Sigma$-entries preceding (and/or following) the renewal time as $N\rightarrow\infty$, w.r.t. the measure $\mu_R$. 

In the next section we study the quantities $K_\alpha^8(N)\in\{0,1,\ldots,7\}$ and $E_\alpha(N)\in\{\pm1\}$ in (\ref{gammaJmodified3}) and our Main Renewal-Type Limit Theorem \ref{theorem:joint-limit-distr} will allow us to include them in the statement about the existence of a joint liming probability distribution. This fact is non trivial since $K_\alpha^8(N)$ and $E_\alpha(N)$ depend on the entire trajectory of $\alpha$ under $R$ until the renewal time $\hat n_N$. 

\subsection{Limiting Distribution for Phase and Conjugation terms 
}\label{section:lim-distr-K8andE}

Let $x_n:=\eta_1\cdots\eta_{\nu_n-1}=\prod_{s=1}^n(-\zeta_s)$ and $y_n:=\kappa_{\nu_n}-1=\sum_{s=1}^n(h_s-\zeta_s)\prod_{u=1}^{s-1}(-\zeta_u)\:\: (\mathrm{mod}\: 8)$. We want to prove that $(x_n,y_n)\in\{\pm1\}\times\{0,1,\ldots,7\}=:\Xi$ have a joint limiting distribution as $n\rightarrow\infty$. We will follow the strategy used by Sinai \cite{Sinai-Topics}, \S 12, to see how the dynamics creates conditional probability distributions and these distributions define uniquely a limiting probability measure. 

Let us consider the natural extension $\hat R:\Sigma^\Z\rightarrow\Sigma^\Z$ of $R$. For $\sigma\in\Sigma^\Z$, denote by $\sigma^-=(\ldots,\sigma_{-2},\sigma_{-1},\sigma_0)$ and $\sigma^+=(\sigma_1,\sigma_2,\ldots)$ and identify the pair $(\sigma^+,\sigma^-)$ with a point in the rectangle $(0,1]\times(-1/3,1]\smallsetminus\Q^2$ as discussed in \cite{Cellarosi}.  
One should notice that the ``past'' is identified with the $y$-axis and the ``future'' with the $x$-axis.
Let us consider cylinders in $\Sigma^\Z$ of the form $J^{(m+1)}_{\sigma_{-n-m},\ldots,\sigma_{-n-1},\sigma_{-n}}$, $n\geq0$, i.e. depending only on the past. Such cylinders $J$ are identified with rectangles $(0,1]\times I$, where $I$ is an interval in the $y$-direction, and by $|J|$ we mean the 1-dimensional Lebesgue measure of $I$.
\begin{lem}\label{lem: creation conditional probabilities for Sigma}
For every $\sigma^-\in\Sigma^\N$, the limit
$$\mu(\sigma_0|\sigma_{-1},\sigma_{-2},\ldots):=\lim_{n\rightarrow\infty}\frac{\left|J^{(n+1)}_{\sigma_{-n},\ldots,\sigma_{-1},\sigma_{0}}\right|}{\left|J^{(n)}_{\sigma_{-n},\ldots,\sigma_{-1}}\right|}$$
exists and satisfies the following conditions:
\begin{eqnarray}
&&\mu(\sigma_0|\sigma_{-1},\ldots)\geq C_{16},
\nonumber\\
&&\sum_{\sigma_0\in\Sigma}\mu(\sigma_0|\sigma_{-1},\ldots)=1,\nonumber\\
&&\left|\frac{\mu(\sigma_0|\sigma_{-1},\ldots,\sigma_{-s},\sigma_{-s-1}',\sigma_{-s-2}',\ldots)}{\mu(\sigma_0|\sigma_{-1},\ldots,\sigma_{-s},\sigma_{-s-1},\sigma_{-s-2},\ldots)}-1\right|\leq C_{17}\,e^{- C_{18}\,s},\label{condition-exp-lemma-creation-cond-prob}
\end{eqnarray}
for some constants $C_{16},C_{17},C_{18}>0$.
\end{lem}
\begin{proof}
Let $l_n=|J^{(n+1)}_{\sigma_{-n},\ldots,\sigma_{-1},\sigma_{0}}|/|J^{(n)}_{\sigma_{-n},\ldots,\sigma_{-1}}|$. 
By Lemma \ref{lem: two-ratios} we have
\begin{eqnarray}
\left|\frac{l_{n+1}}{l_n}-1\right|=\left|\frac{|J^{(n+2)}_{\sigma_{-n-1},\ldots,\sigma_{-1},\sigma_{0}}|}{|J^{(n+1)}_{\sigma_{-n-1},\ldots,\sigma_{-1}}|}\cdot\frac{|J^{(n)}_{\sigma_{-n},\ldots,\sigma_{-1}}|}{|J^{(n+1)}_{\sigma_{-n},\ldots,\sigma_{-1},\sigma_0}|}-1\right|\leq C_{8}\,e^{-C_9\,n}.\nonumber
\end{eqnarray}
This implies the existence of the limit $\lim_{n\rightarrow\infty}l_n$ and also the desired properties.
\end{proof}
Since we are working with the natural extension of $R$, setting $z_n:=h_n-\zeta_n\:\mathrm{(mod\,8)}$, the quantities $(\zeta_n,z_n)\in\Xi$ are defined for every $n\in\Z$. 
Now we want to define conditional probability distributions $\mu_0\left((\zeta_{0},z_{0})\big{|}(\zeta_{-1},z_{-1}),(\zeta_{-2},z_{-2}),\ldots\right)$ over $\Xi^\Z$.
Let us fix a sequence $\underline{\sigma}^{(0)}=\{\sigma_j^{(0)}\}\in\Sigma^\Z$ and for every $n\in\N$ 
consider 
\begin{eqnarray}
&&\mu_0^{(0)}\!\left((\zeta_{0},z_{0})\big{|}(\zeta_{-1},z_{-1}),(\zeta_{-2},z_{-2}),\ldots,(\zeta_{-n},z_{-n})\right)=\nonumber\\
&&=\frac{\mu_0^{(0)}\!\left((\zeta_{-n},z_{-n}),\ldots,(\zeta_{-1},z_{-1}),(\zeta_0,z_0)\right)}{\mu_0^{(0)}\!\left((\zeta_{-n},z_{-n}),\ldots,(\zeta_{-1},z_{-1})\right)}:=\nonumber\\
&&:=\frac{\sum_{\sigma_{0},\sigma_{-1},\ldots,\sigma_{-n}}\mu(\sigma_{-n},\ldots,\sigma_{-1},\sigma_0)}{\sum_{\sigma_{-1},\ldots,\sigma_{-n}}\mu(\sigma_{-n},\ldots,\sigma_{-1})}=\nonumber\\
&&=\frac{\sum_{\sigma_{0},\sigma_{-1},\ldots,\sigma_{-n}}\prod_{s=0}^n\mu(\sigma_{-s}|\sigma_{-s-1},\ldots,\sigma_{-n},\sigma_{-n-1}^{(0)},\sigma_{-n-2}^{(0)},\ldots)}{\sum_{\sigma_{-1},\ldots,\sigma_{-n}}\prod_{s=1}^n\mu(\sigma_{-s}|\sigma_{-s-1},\ldots,\sigma_{-n},\sigma_{-n-1}^{(0)},\sigma_{-n-2}^{(0)},\ldots)}\label{proof-markov-process1},
\end{eqnarray}
where the sums are taken over all possible $\sigma_0,\sigma_{-1},\ldots,\sigma_{-n}\in\Sigma$ which are compatible with the values of $(\zeta_{-n},z_n),\ldots(\zeta_{-1},z_{-1}),(\zeta_0,z_0)$.
\begin{lem}\label{lem: existence mu0}
The limit $$\mu_0\!\left((\zeta_0,z_0)|(\zeta_{-1},z_{-1}),(\zeta_{-2},z_{-2}),\ldots\right):=\lim_{n\rightarrow\infty}\mu_0^{(0)}\!\left((\zeta_0,z_0)|(\zeta_{-1},z_{-1}),(\zeta_{-2},z_{-2}),\ldots,(\zeta_{-n},z_n)\right)$$
exists and does not depend on $\underline{\sigma}^{(0)}$.
\end{lem}
\begin{proof}
The Markov process $\{\ldots,\sigma_{-n},\ldots,\sigma_{-1},\sigma_0\}$ has a countable state-space but, by (\ref{lemmaconstants}), it satisfies a Doeblin condition.
Therefore it can be exponentially well approximated by a process with finite (but sufficiently large) state-space. To this end, let us introduce also $\mu^{(0)}_{0,L}$ as in (\ref{proof-markov-process1}), with the additional constraint that $\sigma_{-j}=h_{-j}\cdot m_{-j}^{\zeta_{-j}}$, satisfy the inequalities $h,m\leq L$ for $0\leq j\leq n$. The sums in the corresponding numerator and denominator are thereby finite and contain at most $(2L^2-L-1)^{n+1}$ and $(2L^2-L-1)^n$ terms respectively.
In order to prove that $\mu^{(0)}_{0,L}\left((\zeta_{0},z_{0})\big{|}(\zeta_{-1},z_{-1}),(\zeta_{-2},z_{-2}),\ldots,(\zeta_{-n},z_{-n})\right)$ has a limit as $n\rightarrow\infty$ we shall perform a second approximation of the process $\{\sigma_j\}$ by a finite Markov chain with memory of order $\sqrt{n}$.

We partition the integers $1,\ldots,n$ into fragments with $\lfloor\sqrt n\rfloor$ elements. Notice that $0\leq n-\lfloor\sqrt n\rfloor^2\leq 2\lfloor\sqrt n\rfloor$ and define 
$$\mathrm{sq}(n)=\begin{cases}
	\lfloor\sqrt n\rfloor-1&\text{if $0\leq n-\lfloor\sqrt n\rfloor^2<\lfloor\sqrt n\rfloor$},\\
      \lfloor\sqrt n\rfloor& \text{if $\lfloor\sqrt n\rfloor\leq n-\lfloor\sqrt n\rfloor^2<2\lfloor\sqrt n\rfloor$}, \\
      \lfloor\sqrt n\rfloor+1 & \text{if $n-\lfloor\sqrt n\rfloor^2= 2\lfloor\sqrt n\rfloor$}.
\end{cases}$$
The product in the denominator of $\mu_{0,L}^{(0)}$ becomes
\begin{eqnarray}
&&\prod_{s=1}^n\mu\left(\sigma_{-s}\Big{|}\sigma_{-s-1},\ldots,\sigma_{-n},\sigma_{-n-1}^{(0)},\sigma_{-n-2}^{(0)},\ldots\right)=\nonumber\\
&&=\prod_{j=1}^{\mathrm{sq}(n)}\mu\left(\sigma_{-(j-1)\lfloor\sqrt n\rfloor-1},\ldots,\sigma_{-j\lfloor\sqrt n\rfloor}\Big{|}\sigma_{-j\lfloor\sqrt n\rfloor-1},\ldots,\sigma_{-(j+1)\lfloor\sqrt n\rfloor},
\ldots,\sigma_{-n},\sigma_{-n-1}^{(0)},
\ldots\right)\cdot\nonumber\\
&&\hspace{.3cm}\cdot\,\mu\left(\sigma_{-\mathrm{sq}(n)\lfloor\sqrt n\rfloor-1},\ldots,\sigma_{-(\mathrm{sq}(n)+1)\lfloor\sqrt n\rfloor}\Big|
\sigma_{-(\mathrm{sq}(n)+1)\lfloor\sqrt n\rfloor-1},\ldots,\sigma_{-n},\sigma_{-n-1}^{(0)},\ldots\right)\cdot\label{factor1}\\
&&\hspace{.3cm}\cdot\,\mu\left(\sigma_{-(\mathrm{sq}(n)+1)\lfloor\sqrt n\rfloor-1},\ldots,\sigma_{-n}\Big|\sigma_{-n-1}^{(0)},\sigma_{-n-2}^{(0)},\ldots\right)=\label{factor0}\\
&&=\left(\prod_{j=1}^{\mathrm{sq}(n)}\mu\left(\hat\sigma_{-j}\big{|}\hat\sigma_{-j-1}\right)\delta_j\right)\cdot\tilde\mu^{(1)}\cdot\tilde\mu^{(0)},\nonumber
\end{eqnarray}
where
\begin{eqnarray}
&&\hat\sigma_{-j}=(\sigma_{-(j-1)\lfloor\sqrt n\rfloor-1},\ldots,\sigma_{-j\lfloor\sqrt n\rfloor})\in\Sigma^{\lfloor\sqrt n\rfloor},\nonumber\\
&&
\delta_j=\frac{\mu\left(\hat\sigma_{-j}|\hat\sigma_{-j-1},\sigma_{-(j+1)\lfloor\sqrt n\rfloor-1},\ldots\right)}{\mu\left(\hat\sigma_{-j}|\hat\sigma_{-j-1}\right)},\label{correction-term}
\end{eqnarray}
and $\tilde\mu^{(1)}$, $\tilde\mu^{(0)}$ correspond the factors in (\ref{factor1}) and (\ref{factor0}) respectively. Notice that for $n-\lfloor\sqrt n\rfloor^2=k\lfloor\sqrt n\rfloor$, $k=0,1,2$, the factor $\tilde\mu^{(0)}$ disappears and $\tilde\mu^{(1)}=\mu(\sigma_{-\mathrm{sq}(n)\lfloor\sqrt n\rfloor-1},\ldots,\sigma_{-n}\big|\sigma_{-n-1}^{(0)},\ldots)$.
We claim that
\begin{equation}\label{lemma-markov-process-estimate}
\left|\delta_j-1\right|\leq C_{19}\sqrt n\,e^{-C_{20}\,\sqrt n}
\end{equation}
In fact, the correction factor $\delta_j$ can be written as
\begin{eqnarray}
\delta_j=\prod_{s=(j-1)\lfloor\sqrt n\rfloor+1}^{j\lfloor\sqrt n\rfloor}\frac{\mu\left(\sigma_{-s}|\sigma_{-s-1},\ldots,\sigma_{-j\lfloor\sqrt n\rfloor},\hat\sigma_{-j-1},\sigma_{-(j+1)\lfloor\sqrt n\rfloor-1},\ldots\right)}{
\mu\left(\sigma_{-s}|\sigma_{-s-1},\ldots,\sigma_{-j\lfloor\sqrt n\rfloor},\hat\sigma_{-j-1}\right)}\label{proof-lemma-markov-process-estimate0}
\end{eqnarray}
and, by (\ref{condition-exp-lemma-creation-cond-prob}), each factor in (\ref{proof-lemma-markov-process-estimate0}), is $(C_{17}\,e^{-C_{18}\,\sqrt n})$-close to 1. Therefore, for some constants $C_{21},C_{22}>0$, 
$\left|\log\delta_j\right|\leq C_{21}\sqrt n\cdot e^{-C_{22}\,\sqrt n}$ and we get (\ref{lemma-markov-process-estimate}) for some $C_{19},C_{20}>0$. 
The factors $\tilde\mu^{(0)}$ and $\tilde\mu^{(1)}$ can be approximated in the same way, by truncating the length of the condition after $\lfloor\sqrt n\rfloor$ digits. Denoting by $\delta^{(l)}=\frac{\tilde\mu^{(l)}}{\hat\mu^{(l)}}$, $l=0,1$, the correction terms as in (\ref{correction-term}), one gets $|\delta^{(l)}-1|\leq C_{22}\sqrt n\, e^{-C_{23}\,\sqrt n}$ for $l=0,1$ and for some $C_{22},C_{23}>0$.

Therefore $\mu^{(0)}_{0,L}\left((\zeta_{0},z_{0})\big{|}(\zeta_{-1},z_{-1}),(\zeta_{-2},z_{-2}),\ldots,(\zeta_{-n},z_{-n})\right)$ is exponentially well approximated by 
$$\frac{\sum_{\sigma_0,\sigma_{-1},\ldots,\sigma_{-n}}\mu(\sigma_0|\sigma_{-1})\prod_{j=1}^{\mathrm{sq}(n)}\mu(\hat\sigma_{-j}|\hat\sigma_{-j-1})\cdot\hat\mu^{(1)}\hat\mu^{(0)}}{\sum_{\sigma_{-1},\ldots,\sigma_{-n}}\prod_{j=1}^{\mathrm{sq}(n)}\mu(\hat\sigma_{-j}|\hat\sigma_{-j-1})\cdot\hat\mu^{(1)}\hat\mu^{(0)}},$$
which can be understood as the expectation of $\mu(\sigma_0|\sigma_{-1})$ with respect to the measure for the finite Markov chain $\{\ldots,\hat\sigma_{-n},\ldots,\hat\sigma_{-1}\}$. Recall that the phase-space of such Markov chain is $\{h\cdot m^\zeta\in\Sigma :\:h,m\leq L\}^{\lfloor\sqrt n\rfloor}$, which has $(2L^2-L-1)^{\lfloor\sqrt n\rfloor}$ elements. This Markov chain is ergodic because, by the symbolic coding of the map $R$, every sequence of elements of $\Sigma$ is allowed.
By the ergodic theorem for Markov chains and the Doeblin condition we get the existence of the limit 
\begin{eqnarray}
&&\mu^{(0)}_{0}\!\left((\zeta_{0},z_{0})\big{|}(\zeta_{-1},z_{-1}),(\zeta_{-2},z_{-2}),\ldots\right)=\nonumber\\
&&=\lim_{n\rightarrow\infty} \lim_{L\rightarrow\infty} \mu^{(0)}_{0,L}\!\left((\zeta_{0},z_{0})\big{|}(\zeta_{-1},z_{-1}),(\zeta_{-2},z_{-2}),\ldots,(\zeta_{-n},z_{-n})\right).\nonumber
\end{eqnarray}
Moreover, by (\ref{condition-exp-lemma-creation-cond-prob}), the conditional probability distributions $\mu_{0}^{(0)}\!\left((\zeta_{0},z_{0})\big{|}(\zeta_{-1},z_{-1}),
\ldots\right)$ do not depend on the sequence $\underline\sigma^{(0)}$ and will be denoted simply by $\mu_{0}\!\left((\zeta_{0},z_{0})\big{|}(\zeta_{-1},z_{-1}),
\ldots\right)$.
\end{proof}

Now, let us fix an arbitrary sequence $\big\{(\zeta_j^{(0)},z_j^{(0)})\big\}_{j\in\Z}\in\Xi^\Z$. For each $s\in\Z$ consider the measure $\lambda_s^{(0)}$ defined on $\Xi^\Z$ using Lemma \ref{lem: existence mu0} as follows:
\begin{eqnarray}
&&\lambda_s^{(0)}\big\{(\zeta_{s-n}^{(0)},z_{s-n}^{(0)}),\ldots,(\zeta_{s-1}^{(0)},z_{s-1}^{(0)})\big\}:=1\hspace{.5cm}\mbox{for every $n\in\N$};\nonumber\\
&&\lambda_s^{(0)}\big\{(\zeta_{s},z_{s}),(\zeta_{s+1},z_{s+1}),\ldots,(\zeta_{s+t},z_{s+t})\big\}:=\nonumber\\
&&\hspace{1cm}:=\prod_{l=s}^{s+t}\mu_0\!\left((\zeta_{l},z_{l})\Big{|}(\zeta_{l-1},z_{l-1}),\ldots,(\zeta_{s},z_{s}),(\zeta_{s-1}^{(0)},z_{s-1}^{(0)}),(\zeta_{s-2}^{(0)},z_{s-2}^{(0)}),\ldots\right)\nonumber
\end{eqnarray}
for every $t\geq0$. Since $\Xi^\Z$ is compact, the space of all probability measures on it is weakly compact and therefore there exists a subsequence $\{-s_j\}_{j\in\N}$ such that $\lim_{j\rightarrow\infty}s_j=\infty$ and $\lambda_{-s_j}^{(0)}\Longrightarrow\lambda^{(0)}$ as $j\rightarrow\infty$. One can show (see \cite{Sinai-Topics}, \S12, Theorem 2 and Lemma 2) that
$$\lim_{n\rightarrow\infty}\lambda^{(0)}\left((\zeta_s,z_s)\big{|}(\zeta_{s-1},z_{s-1}),\ldots,(\zeta_{s-n},z_{s-n})\right)=\mu_0\left((\zeta_{s},z_{s})\big{|}(\zeta_{s-1},z_{s-1}),(\zeta_{s-2},z_{s-2}),\ldots\right)$$
and such $\lambda^{(0)}$ is shift-invariant and unique.

Let us now prove the existence of the limiting probability distribution for the sequence $\{(x_n,y_n)\}_{n\in\N}$. Observe that 
\begin{eqnarray}
&x_1=-\zeta_1,\:&x_n=x_{n-1}\cdot(-\zeta_n);\nonumber\\
&y_1=z_1,\:&y_n=y_{n-1}+z_n\cdot x_{n-1}.\nonumber
\end{eqnarray}
\begin{lem}
For every $(X,Y)\in\Xi$ the limit 
$$\lim_{n\rightarrow\infty}\scriptsize{\lambda^{(0)}\!\left(\!\!\begin{array}{c}x_n=X\\y_n=Y\end{array}\!\!\right)}$$
exists. 
\end{lem}
\begin{proof}
Using the above relations we get
\begin{eqnarray}
&&\scriptsize{\lambda^{(0)}\!\left(\!\!\begin{array}{c}x_n=X\\y_n=Y\end{array}\!\!\right)=\sum_{\tiny{\begin{array}{c}X_{n-1},\ldots,X_1\\Y_{n-1},\ldots,Y_1\end{array}}}
\prod_{j=1}^{n-1}\lambda^{(0)}\!\left(\!\!\begin{array}{c}x_{j+1}=X_{j+1}\\y_{j+1}=Y_{j+1}\end{array}\!\!\Big|\!\!\begin{array}{c}x_{j}=X_{j}\\y_{j}=Y_{j}\end{array}\!\!\right)\cdot\lambda^{(0)}\!\left(\!\!\begin{array}{c}x_{1}=X_{1}\\y_{1}=Y_{1}\end{array}\!\!\right)=}\nonumber\\
&&\scriptsize{=\sum_{\tiny{\begin{array}{c}X_{n-1},\ldots,X_1\\Y_{n-1},\ldots,Y_1\end{array}}}
\prod_{j=1}^{n-1}\lambda^{(0)}\!\left((\zeta_{j+1},z_{j+1})=Z_{j+1}\big|(\zeta_{j},z_{j})=Z_j\right)\cdot\lambda^{(0)}\!\left((\zeta_{1},z_{1})=Z_{1}\right)
,\label{sum-product-XY}}
\end{eqnarray}
where $(X_n,Y_n):=(X,Y)$, 
$(X_{n-1},Y_{n-1}),\ldots,(X_1,Y_1)\in\Xi$ 
and $Z_j\in\Xi$ are defined as
\begin{equation}\label{eq: trans-XY-Z}
Z_1:=(-X_1,Y_1),\hspace{.6cm}Z_j:=\big(-X_{j-1}X_j,X_{j-1}(Y_j-Y_{j-1})\:(\mathrm{mod}\:8)\big),\:\:j\geq 2.
\end{equation}
Notice that, by (\ref{eq: trans-XY-Z}), the sum over all $X_1,\ldots,X_{n-1},Y_1,\ldots,Y_{n-1}$ in (\ref{sum-product-XY}) can be replaced by the sum over all possible $Z_1,\ldots,Z_{n-1}\in\Xi$. 

Let us denote by $p_{Z,W}:=\lambda^{(0)}\!\left((\zeta_{j+1},z_{j+1})=W\big|(\zeta_{j},z_{j})=Z\right)$, the transition probabilities for  $Z,W\in\Xi$, by $\Pi:=(p_{Z,W})_{Z,W\in\Xi}$ the corresponding $2^4\times 2^4$ stochastic matrix and by $\underline\pi:=\left(\lambda^{(0)}\!\left((\zeta_{1},z_{1})=Z\right)
\right)_{Z\in\Xi}$ the initial probability distribution. Thus, we can write (\ref{sum-product-XY}) as
\begin{eqnarray}\label{P0(xn,yn) as matrix power}
\scriptsize{\lambda^{(0)}\!\left(\!\!\begin{array}{c}x_n=X\\y_n=Y\end{array}\!\!\right)}=\left(\Pi^n\underline\pi\right)_Z,
\end{eqnarray}
where $Z=\big(-X_{j-1}X_j,X_{j-1}(Y_j-Y_{j-1})\:(\mathrm{mod}\:8)\big)$.
The stochastic matrix $\Pi$ has positive entries and therefore $\scriptsize{\lambda^{(0)}\!\left(\!\!\begin{array}{c}x_n=X\\y_n=Y\end{array}\!\!\right)}$ has a limit for every $(X,Y)\in\Xi$ as $n\rightarrow\infty$.
\end{proof}

Let $J$ be as in the previous section. It represents a finite number of $\Sigma$-entries preceding the renewal time $\hat n_N$ defining the approximating curve $t\mapsto \gamma_{\alpha,N}^J(t)$.
We can rewrite $E_\alpha(N)$ and $K_\alpha^8(N)$ as follows:
\begin{eqnarray}
E_\alpha(N)&=&x_{\hat n_N-J}\cdot\mathcal{E}_J
,\nonumber\\
K_\alpha^8(N)&=&\left[1+y_{\hat n_N-J}+x_{\hat n_N-J}\cdot\!\!\!\sum_{u=\hat n_N-J+1}^{\hat n_N}\!\!\!(h_u-\zeta_u)\,\mathcal{E}_J^{\hat n_N-u}\right]_8,
\nonumber\\
\left(E_\alpha(N),K_\alpha^8(N)\right)&=&\mathrm F_2\!\left((x_{\hat n_N-J},y_{\hat n_N-J}),\{h_l\cdot m_l^{\zeta_l},\hat n_N-J <l\leq \hat n_N\}\right),\label{EK8=F2}
\end{eqnarray}
where $\mathrm F_2:\Xi\times\Sigma^J\rightarrow\Xi$.

\section{Existence of Limiting Finite-Dimensional Distributions
}\label{section-limiting-fin-dim-distr}
In this section we prove the existence of limiting finite-dimensional distribution for $\gamma_{\alpha,N}^J$ as $N\rightarrow\infty$, w.r.t. $\mu_R$. Thereafter, we extend the result to $\gamma_{\alpha,N}$. We also discuss the notion of \emph{nice} set and we give a sufficient condition for a set $A\subset \C^k$ to be nice.

For every $t\in[0,1]$, by (\ref{gammaJmodified3}) and (\ref{EK8=F2}), we can write
\begin{equation}\nonumber
\gamma_{\alpha,N}^J(t)=\mathrm F\!\left(t;R^{\hat n_N}(\alpha),\frac{\hat q_{\hat n_N-1}}{N},\frac{\hat q_{\hat n_N}}{N},(x_{\hat n_N-J},y_{\hat n_N-J}),\left\{\sigma_l\right\}_{l=\hat n_N-J}^{\hat n_N}\right),
\end{equation}
where $\mathrm F=\mathrm F^{(1)}:[0,1]\times(0,1]\times(0,1]\times(1,\infty)\times\Xi\times\Sigma^J\rightarrow\C$ is a 
measurable 
function of its arguments.
Similarly, for every $0\leq t_1<t_2<\cdots<t_k\leq1$, setting $\underline{\gamma}_{\alpha,N}^J(t_1,\ldots,t_k):=(\gamma_{\alpha,N}^J(t_1),\ldots,\gamma_{\alpha,N}^J(t_k))$, we have
$$\underline\gamma_{\alpha,N}^J(t_1,\ldots,t_k)=\mathrm F^{(k)}\!\left(\!(t_1,\ldots,t_k);R^{\hat n_N}(\alpha),\frac{\hat q_{\hat n_N-1}}{N},\frac{\hat q_{\hat n_N}}{N},(x_{\hat n_N-J},y_{\hat n_N-J}),\left\{\sigma_l\right\}_{l=\hat n_N-J}^{\hat n_N}\right),$$
where $\mathrm F^{(k)}:[0,1]^k\times(0,1]\times(0,1]\times(1,\infty)\times\Xi\times\Sigma^J\rightarrow\C^k$.

The following Renewal-Type Limit Theorem is the core of the proof of the existence of finite-dimensional distributions for $\gamma_{\alpha,N}^J$ as $N\rightarrow\infty$. It is a generalization of Theorem 1.6 in \cite{Cellarosi} and its proof will be sketched in Appendix \ref{AppendixA}. Let us just mention that it relies on the mixing property of the special flow built over the natural extension of $R$, under the a suitably chosen roof function.
\begin{theorem}[Main Renewal-Type Limit Theorem]\label{theorem:joint-limit-distr}
Fix $N_1,N_2\in\N$. The quantities $\frac{\hat q_{\hat n_N-1}}{N}$, $\frac{\hat q_{\hat n_N}}{N}$, $\{\sigma_{\hat n_N+l}\}_{l=-N_1+1}^{N_2}$, $(x_{\hat n_N-N_1}$, $y_{\hat n_N-N_1})$ have a joint limiting probability distribution w.r.t. the measure $\mu_R$ as $N\rightarrow\infty$.

In other words: there exists a probability measure $\mathrm{Q}=\mathrm{Q}_{N_1,N_2}$ on the space $(0,1]\times(1,\infty)\times\Sigma^{N_1+N_2}\times\Xi$ such that for every $a_1,b_1,a_2,b_2\in\R$, $0<a_1<b_1<1<a_2<b_2$, for every $\underline{c}=(c_l)_{l=-N_1+1}^{N_2}\in\Sigma^{N_1+N_2}$ and for every $(x,y)\in\Xi$, we have
\begin{eqnarray}
&&\hspace{-1.2cm}\mu_R\left(\left\{\alpha:\:a_1<\frac{\hat q_{\hat n_N-1}}{N}<b_1,\:a_2<\frac{\hat q_{\hat n_N}}{N}<b_2,\:(\sigma_{\hat n_N+l})_{l=-N_1+1}^{N_2}=\underline{c},\:\scriptsize{\left(\!\!\begin{array}{l}x_{\hat n_N-N_1}\\y_{\hat n_N-N_1}\end{array}\!\!\right)\!=\!\left(\!\!\begin{array}{l}x\\y\end{array}\!\!\right)}\!\right\}\!\right)\label{eq:thm-joint-lim}\\
&&\hspace{-1.2cm}\longrightarrow\nonumber\mathrm{Q}\!\left((a_1,b_1)\times(a_2,b_2)\times\{\underline{c}\}\times\{(x,y)\}\right)\hspace{.3cm}\mbox{as $N\rightarrow\infty$}.
\end{eqnarray}
\end{theorem}
\begin{remark}\label{rk: shrinking sets}
Let us also mention that the proof of Theorem \ref{theorem:joint-limit-distr} provides an explicit formula for $\mathrm{Q}\left((a_1,b_1)\times(a_2,b_2)\times\{\underline{c}\}\times\{(x,y)\}\right)$, based on a geometrical construction.
Moreover, if we fix $\underline c\in\Sigma^{N_1+N_2}$ and $(x,y)\in\Xi$, then the measure on $(0,1]\times(1,\infty)$ defined as $\mathrm Q_{N_1,N_2;\underline c,(x,y)}(E):=\mathrm Q_{N_1,N_2}(E\times\{\underline c\}\times\{(x,y)\})$ is equivalent to the Lebesgue measure on $(0,1]\times(1,\infty)$.
\end{remark}
Notice that the limiting probability distribution of $R^{\hat n_N}(\alpha)=(\sigma_{\hat n_N+1},\sigma_{\hat n_N+2},\ldots)\in\Sigma^\N$ can be obtained by providing a limiting probability distribution for any fixed number of $\Sigma$-entries after the renewal time $\hat n_N$, i.e. $\sigma_{\hat n_N+1},\ldots,\sigma_{\hat n_N+N_2}$, $N_2\in\N$. We immediately get the following
\begin{cor}\label{cor: QJ}
Fix $J\in\N$. The quantities $R^{\hat n_N}$, $\frac{\hat q_{\hat n_N-1}}{N}$, $\frac{\hat q_{\hat n_N}}{N}$, $(x_{\hat n_N-J}$, $y_{\hat n_N-J})$, $\{\sigma_l\}_{l=\hat n_N-J}^{\hat n_N}$ have a joint limiting probability distribution on $(0,1]\times(0,1]\times(1,\infty)\times\Xi\times\Sigma^{J+1}$ as $N\rightarrow\infty$, with respect to the measure $\mu_R$ on $(0,1]$.
\end{cor} 
Let us denote the limiting probability measure by $\mathrm Q^{(J)}$. For every $(x,y)\in\Xi$ and $\underline\sigma\in\Sigma^{J+1}$ the measure on $(0,1]^2\times(1,\infty)$ defined as $\mathrm Q^{(J)}_{(x,y),\underline\sigma}(E):=\mathrm Q^{(J)}\big(E\times\{(x,y)\}\times\{\underline\sigma\}\big)$ 
is equivalent to the Lebesgue measure on $(0,1]^2\times(1,\infty)$. This fact is a consequence of Remark \ref{rk: shrinking sets}.

\begin{remark}
Fix $(t_1,\ldots,t_k)\in[0,1]^k$, $J\in\N$, $(x,y)\in\Xi$ and $\underline\sigma\in\Sigma^{J+1}$. Denoting $(u,v,w)=\left(R^{\hat n_N}(\alpha),\frac{\hat q_{\hat n_N-1}}{N},\frac{\hat q_{\hat n_N}}{N}\right)$, we can rewrite the functions in Lemma \ref{lemma: partial products} as
\begin{eqnarray}\beta_j=\beta_j(u)=\frac{a_j^{(1)}+b_j^{(1)} u}{c_j^{(1)}+d_j^{(1)}u},\:\:B_{s,j}=B_{s,j}(u)=\frac{a_{s,j}^{(2)}+b_{s,j}^{(2)} u}{c_{s,j}^{(2)}+d_{s,j}^{(2)}u},\:\:
D_j=D_j(u)=\prod_{l=0}^{j-1}\frac{a_l^{(3)}+b_l^{(3)} u}{c_l^{(3)}+d_l^{(3)}u},\nonumber\end{eqnarray}
for some constants $a_j^{(1)},b_j^{(1)},c_j^{(1)},d_j^{(1)},a_{s,j}^{(2)},b_{s,j}^{(2)},c_{s,j}^{(2)},d_{s,j}^{(2)},a_l^{(3)},b_l^{(3)},c_l^{(3)},d_l^{(3)}$ (determined by $\underline\sigma$). Notice that the functions $\beta_j$, $B_{s,j}$ and $D_j$ take values in $(0,1]$ and, despite their rational structure, they are $\mathcal C^\infty$ functions of $u\in(0,1]$.
Moreover $\alpha_{\nu_{\hat n_N-1}}=\alpha_{\nu_{\hat n_N-1}}(u)=\frac{a^{(4)}+b^{(4)} u}{c^{(4)}+d^{(4)}u}\in(0,1]$, by (\ref{proof: limiting distribution for Theta 1}) and (\ref{proof: limiting distribution for Theta 2}), 
\begin{eqnarray}
\Theta_\alpha(N)&=:&\theta(u,v,w)=\left(a^{(5)}v+b^{(5)} w+c^{(5)}\alpha_{\nu_{\hat n_N-1}}\left(d^{(5)}v+e^{(5)}w\right)\right)^{-1}=\nonumber\\
&=&\frac{c^{(4)}+d^{(4)}u}{\left(a^{(5)}v+b^{(5)}w\right)\left(c^{(4)}+d^{(4)}u\right)+c^{(5)}\left(a^{(4)}+b^{(4)} u\right)\left(d^{(5)}v+e^{(5)}w\right)}\in(0,\infty)\nonumber
\end{eqnarray}
is also a $\mathcal C^\infty$ function of $(u,v,w)$, where $a^{(4)},b^{(4)},c^{(4)},d^{(4)},a^{(5)},b^{(5)},c^{(5)},d^{(5)},e^{(5)}$ are some constants (determined by $\underline \sigma$).
For $\underline t=(t_1,\ldots,t_k)$, set $$f_{\underline t}^{(J)}:=\mathrm F^{(k)}\big((t_1,\ldots,t_k),\cdot\big):(0,1]^2\times(1,\infty)\times\Xi\times\Sigma^{J+1}\rightarrow\C^k.$$ 
Finally, $\alpha_{\nu_{\hat n_N-j}}=:A_j(u)=\frac{a_j^{(6)}+b_j^{(6)} u}{c_j^{(6)}+d_j^{(6)}u}\in(0,1]$ and
$$
f^{(J)}_{\underline t;(x,y),\underline\sigma}:=\mathrm F^{(k)}\!\left((t_1,\ldots,t_k);\cdot,(x,y),\underline\sigma\right)=f^{(J)}_{\underline t}\!\left(\cdot,(x,y),\underline\sigma\right):(0,1]^2\times(1,\infty)\rightarrow\C^k$$
reads as
\begin{eqnarray}f^{(J)}_{\underline t;(x,y),\underline\sigma}(u,v,w)&=& \Bigg(C^{(1)}\theta(u,v,w)^{-\ha}\bigg[ C^{(2)} \mathcal{S}_{u}^{(C^{(3)})}(t_l\,\theta(u,v,w))+\nonumber\\
&&+\sum_{j=0}^{J-1}D_j(u)^\ha\!\sum_{s=2}^{C^{(4)}_{j}+2}C^{(5)}_s B_{s,j}(u)^\ha\,{\Gamma}\Big(A_j(u),t_l\frac{\theta(u,v,w)}{B_{s,j}(u)D_j(u)}\Big)\bigg]^{(C^{(6)})}\Bigg)_{l=1}^{k},\nonumber 
\end{eqnarray}
where 
and $C^{(1)},C^{(2)},C^{(5)}_s\in\C$, $C^{(3)},C^{(6)}\in\{\pm1\}$ and $C^{(4)}_j\in\N$ are constants determined by $(x,y)\in\Xi$ and $\underline\sigma\in\Sigma^{J+1}$.
Notice that 
$f_{\underline t;(x,y),\underline\sigma}^{(J)}:(0,1]^2\times(1,\infty)\rightarrow\C^k$ a continuous function (with piecewise $\mathcal C^\infty$ partial derivatives) of $(u,v,w)$. 
\end{remark}

\subsection{Nice sets}
We say that $A\in\mathcal B^k$ is \emph{$(t_1,\ldots,t_k)$-nice} (or simply \emph{nice}) if for every $J\in\N$, for every $(x,y)\in\Xi$ and every $\underline\sigma\in\Sigma^{J+1}$, $\partial\big((f^{(J)}_{\underline t;(x,y),\underline\sigma})^{-1}(A)\big)$ has zero Lebesgue measure in $(0,1]^2\times(0,\infty)$. 

Notice that if $A=A_1\times\ldots\times A_k$, where $A_l\in\mathcal{B}^1$ and $A_l$ is $t_l$-nice for $l=1,\ldots,k$, than $A$ is $(t_1,\ldots,t_k)$-nice. The following Lemma gives a sufficient condition for $A\in\mathcal B^1$ to be $t$-nice
, analogous to Lemma 5.1 in \cite{Marklof1999b}.
\begin{lem}
Let $A\in\mathcal B^1$ be an open convex set, $0\in A$, with smooth boundary. Let $A(w,\rho):=\{\rho z+w:\:z\in A\}$. Fix $t\in[0,1]$ and $w\in\C$. Then, except for countably many $\rho$, $A(w,\rho)$ is $t$-nice.
\end{lem}
\begin{proof}
Let $t\in[0,1]$ be fixed. For every $J\in\N$, every $(x,y)\in\Xi$ and every $\underline \sigma\in\Sigma^{J+1}$ the set $(0,1]^2\times(1,\infty)
$ has finite $\mathrm Q^{(J)}_{(x,y),\underline\sigma}$-measure, say $
q^{(J)}_{(x,y),\underline\sigma}>0$. Since $f^{(J)}_{t;(x,y),\underline\sigma}$ is measurable, the measure of the set $\mathcal{X}(\rho)=\{(u,v,w)\in(0,1]^2\times(1,\infty):\:f^{(J)}_{t;(x,y),\underline\sigma}(u,v,w)\in A(w,\rho)\}$ tends to $q^{(J)}_{(x,y),\underline\sigma}$ as $\rho\rightarrow\infty$. Since $A(w,\rho)$ is convex for every $\rho$, the sets $\mathcal{I}(\rho)=\{(u,v,w)\in(0,1]^2\times(1,\infty):\:f^{(J)}_{t}\in\partial A(w,\rho)\}$ are  disjoint for different values of $\rho$. Therefore, there can be only countably many $\rho$ for which $\mathcal I(\rho)$ has positive $\mathrm Q^{(J)}_{(x,y),\underline\sigma}$ (and thus Lebesgue) measure. Since $f_{t;(x,y),\underline\sigma}^{(J)}$ is continuous
, the boundary of $\mathcal X(\rho)$ is contained in $\mathcal I(\rho)$, concluding thus the proof.
\end{proof}
\subsection{Limiting Finite-Dimensional Distributions for $\gamma_{\alpha,N}^J$ and $\gamma_{\alpha,N}$}
The main consequence of our Main Renewal-Type Limit Theorem \ref{theorem:joint-limit-distr} is the following
\begin{prop}[Limiting finite dimensional distributions for $\gamma_{\alpha,N}^J$]\label{cor: limiting finite dim distr gammaJ}
For every $k\in\N$ and every $0\leq t_1<t_2<\cdots<t_k\leq1$ there exists a probability measure $\mathrm P_{t_1,\ldots,t_k}^{(J,k)}$ on $\C^k$ such that for every 
open, $(t_1,\ldots,t_k)$-nice set $A\in\mathcal{B}^{k}$, 
 we have
\begin{eqnarray}
\lim_{N\rightarrow\infty}\mu_R\left(\left\{\alpha\in(0,1]:\:\underline \gamma_{\alpha,N}^J(t_1,\ldots,t_k)\in A\right\}\right)=\mathrm P_{t_1,\ldots,t_k}^{(J,k)}(A).\label{limiting finite dim distr gammaJ}
\end{eqnarray}
Moreover, if $\{A^{(j)}\}_{j\in\N}$, $A^{(j)}\in\mathcal B^k$, 
 is a decreasing sequence of open, $(t_1,\ldots,t_k)$-nice sets such that $\mathrm{Leb}(A_j)\rightarrow0$, then $\lim_{j\rightarrow\infty} \mathrm P^{(J,k)}_{t_1,\ldots,t_k}(A^{(j)})=0$.
\end{prop}
\begin{proof}
Since 
$A\in\mathcal{B}^k$ is open and $(t_1,\ldots,t_k)$-nice,
the set $\big\{\alpha\in(0,1]:\:\underline\gamma_{\alpha,N}^J(t_1,\ldots,t_k)\in A\big\}$ can be written as 
\begin{equation}\label{proof limiting finite dim distr gammaJ 1}\left\{\alpha:\:\left(R^{\hat n_N},\frac{\hat q_{\hat n_N-1}}{N},\frac{\hat q_{\hat n_N}}{N},(x_{\hat n_N-J},y_{\hat n_N-J}),\{\sigma_l\}_{l=\hat n_N-J}^{\hat n_N}\right)\in (f^{(J)}_{\underline t})^{-1}(A)\right\}
\end{equation}
and 
\begin{eqnarray}
(f_{\underline t}^{(J)})^{-1}(A)=
\bigsqcup_{\scriptsize{\begin{array}{l} 
(x,y)\in\Xi\\ \underline\sigma\in\Sigma^{J+1}\end{array}}}\!\!\!\!B_{(x,y),\underline\sigma}\times\{(x,y)\}\times\{\underline\sigma\}=\nonumber
\!\!\bigsqcup_{\scriptsize{\begin{array}{c} (x,y)\in\Xi\\ \underline\sigma\in\Sigma^{J+1},\,l\in\N\\ B_{(x,y),\underline\sigma}\neq\emptyset\end{array}}}
\!\!\!\!R_{(x,y),\underline\sigma}^{(l)}\times\{(x,y)\}\times\{\underline\sigma\},\nonumber
\end{eqnarray}
where $B_{(x,y),\underline\sigma}=B_{(x,y),\underline\sigma}(A):=
\left(f_{\underline t;(x,y),\underline\sigma}^{(J)}\right)^{-1}\!(A)$ are open (possibly empty) subsets of $(0,1]^2\times(1,\infty)$ with boundaries of measure zero and
$R_{(x,y),\underline\sigma}^{(l)}=R_{(x,y),\underline\sigma}^{(l)}(A)\subseteq(0,1]^2\times(1,\infty)$ are parallelepipeds of the form $\big(a_{0},b_{0}\big)\times\big(a_{1},b_{1}\big)\times\big(a_{2},b_{2}\big)$ (the endpoints in each coordinate can be either included or not for different values of $(x,y)$ and $\underline \sigma$) and $a_0,b_0$, $a_1,b_1$, $a_2,b_2$, depend on $(x,y)$, $\underline\sigma$ and $l$.
Thus the set in (\ref{proof limiting finite dim distr gammaJ 1}) is a disjoint union of sets of the form\footnote{Strict inequalities are replaced by \virg{$\leq$} when the endpoints are included.}
$$\left\{\alpha:\:a_0< R^{\hat n_N}<b_0,a_1<\frac{\hat q_{\hat n_N-1}}{N}<b_1,a_2<\frac{\hat q_{\hat n_N}}{N}<b_2,(x_{\hat n_N-J},y_{\hat n_N-J})=(x,y),\{\sigma_l\}_{l=\hat n_N-J}^{\hat n_N}=\underline\sigma
\right\}$$
whose $\mu_R$-measures converge to $\mathrm Q^{(J)}\big(R_{(x,y),\underline\sigma}^{(l)}
\times\{(x,y)\}\times\{\underline\sigma\}\big)$ as $N\rightarrow\infty$ by Corollary \ref{cor: QJ}. This concludes the proof of Proposition \ref{cor: limiting finite dim distr gammaJ} setting $$\mathrm P^{(J,k)}_{t_1,\ldots,t_k}(A):=\!\!\!\!\sum_{\scriptsize{\begin{array}{c} (x,y)\in\Xi\\ \underline\sigma\in\Sigma^{J+1},\,l\in\N\\ B_{(x,y),\underline\sigma}\neq\emptyset\end{array}}}\!\!\!\!\!\!\mathrm Q^{(J)}\big(R_{(x,y),\underline\sigma}^{(l)}(A)\big).$$
\end{proof}

Now, for fixed $k$ and $t_1,\ldots,t_k$ we want to consider the limit of $\mathrm P_{t_1,\ldots,t_k}^{(J,k)}(A)$ as $J\rightarrow\infty$. We have the following
\begin{lem}\label{lemma limPJ}
For every $k\in\N$, every $0\leq t_1<t_2<\ldots<t_k\leq 1$ and every open, $(t_1,\ldots,t_k)$-nice set $A\in\mathcal B^k$, the limit $\lim_{J\rightarrow\infty}\mathrm P_{t_1,\ldots,t_k}^{(J,k)}(A)$
exists. It will be denoted by $\mathrm P_{t_1,\ldots,t_k}^{(k)}(A)$.
\end{lem}
\begin{proof} For simplicity, write $X_N^J(\alpha)=\underline\gamma_{\alpha,N}^J(t_1,\ldots,t_k)$, $X_N(\alpha)=\underline \gamma_{\alpha,N}(t_1,\ldots,t_k)$, $\mu=\mu_R$ and $\mathrm P^{J}=\mathrm P^{(J,k)}_{t_1,\ldots,t_k}$. Moreover, for $z=(z_1,\ldots,z_k)\in\C^k$ set $|z|:=|z_1|+\ldots+|z_k|$. Assume, by contradiction, that the sequence $\{\mathrm P
^{J
}\}_{J\in\N}$ does not have a limit as $J\rightarrow\infty$. In this case there exist $\varepsilon>0$ and a subsequence $\mathcal J=\{J_l\}_{l\in\N}$ such that $|\mathrm P^{J'}(A)-\mathrm P^{J''}(A)|>\varepsilon$ for every $J',J''\in\mathcal J$. By definition of $\mathrm P^{J'}(A)$ and $\mathrm P^{J''}(A)$ we have that for every $\delta_5>0$ and for sufficiently large $N$, 
\begin{equation}\label{eq: lemma limPJ-1}
\left|\mu\left\{X_N^{J'}\in A\right\}-\mu\left\{X_N^{J''}\in A\right\}\right|\geq 1-\delta_5.
\end{equation}
On the other hand, by Lemma \ref{localizinglemma1}, we know that 
\begin{equation}\label{eq: lemma limPJ-2}
\mu\left\{\left|X_N-X_N^J\right|\leq k e^{-C_{15} J}\right\}\geq1-\delta_3(J)
\end{equation}
and $\delta_3(J)\rightarrow0$ as $J\rightarrow\infty$.
Now (\ref{eq: lemma limPJ-2}) implies that
\begin{equation}\nonumber
\mu\left\{\left|X_N^{J'}-X_N^{J''}\right|\leq k(e^{-C_{15}J'}+e^{-C_{15} J''})\right\}\geq 1-\delta_3(J')-\delta_3(J'')
\end{equation}
and thus
\begin{eqnarray}
&&\hspace{-.6cm}\left|\mu\left\{X_N^{J'}\in A\right\}-\mu\left\{X_N^{J''}\in A\right\}\right|\leq\nonumber\\
&&\hspace{-.6cm}\leq\left|\mu\left\{X_N^{J'}\in A,\,\left|X_N^{J'}-X_N^{J''}\right|\leq k(e^{-C_{15} J'}+e^{-C_{15} J''})\right\}-\mu\left\{X_N^{J''}\in A\right\}\right|+\delta_3(J')+\delta_3(J'')\leq\nonumber\\
&&\hspace{-.6cm}\leq \left|\mu\left\{X_N^{J''}\in A'\right\}-\mu\left\{X_N^{J''}\in A\right\}\right|+\delta_3(J')+\delta_3(J'')\label{eq: lemma limPJ-4}
\end{eqnarray}
where $A'=\{z\in \C^k:\:|z-w|\leq k(e^{-C_{15} J'}+e^{-C_{15} J''}), w\in A\}$. Now, by taking sufficiently large $J',J''\in\mathcal J$ and using the last part of Proposition \ref{cor: limiting finite dim distr gammaJ}, (\ref{eq: lemma limPJ-4}) gives
\begin{eqnarray}
\left|\mu\left\{X_N^{J'}\in A\right\}-\mu\left\{X_N^{J''}\in A\right\}\right|\leq \mu\left\{X_N^{J''}\in A'\smallsetminus A\right\}+\delta_3(J')+\delta_3(J'')\leq \varepsilon/3\nonumber,
\end{eqnarray}
contradicting thus (\ref{eq: lemma limPJ-1}) if we choose $\delta_5=\varepsilon/2$.
\end{proof}
Now we can prove our Main Theorem.
\begin{proof}[Proof of Theorem \ref{main-theorem}]
So far, by Lemma \ref{lemma limPJ}, we know that $$\lim_{J\rightarrow\infty}\lim_{N\rightarrow\infty}\mu_R\left\{\alpha:\:\underline \gamma_{\alpha,N}^J(t_1,\ldots,t_k)\in A\right\}=\mathrm P^{(k)}_{t_1,\ldots, t_k}(A).$$
Roughly speaking, we want to interchange the order of the two limits. Let us use the same notations of the proof of Lemma  \ref{lemma limPJ} and, in addition, set $Y_N^J(\alpha):=X_N(\alpha)-X_N^J(\alpha)$ and $\mathrm P:=\mathrm P^{(k)}_{t_1,\ldots,t_k}$. 
By (\ref{eq: lemma limPJ-2}) we have
\begin{eqnarray}
\mu\left\{X_N\in A\right\}\leq \mu\left\{X_N^J+Y_N^J\in A,\:|Y_N^J|\leq k e^{-C_{15} J}\right\}+\delta_3(J)\leq\mu\left\{X_N^J\in A'\right\}+\delta_3(J),\label{eq: pf main theorem 1}
\end{eqnarray}
where $A'=\{z\in\C^k:\:|z-w|\leq k e^{-C_{15} J},\:w\in A\}$ and $\delta_3(J)\rightarrow0$ as $J\rightarrow\infty$. Now, by Proposition \ref{cor: limiting finite dim distr gammaJ} and Lemma \ref{lemma limPJ}, (\ref{eq: pf main theorem 1}) becomes
\begin{eqnarray}
\mu\left\{X_N\in A\right\}\leq P^J(A)+\delta_6(N)+\delta_3(J)=P(A)+\delta_7(J)+\delta_6(N)+\delta_3(J),\label{eq: pf main theorem 2}
\end{eqnarray}
where $\delta_6(N)\rightarrow0$ as $N\rightarrow\infty$ and $\delta_7(J)\rightarrow0$ as $J\rightarrow\infty$.
On the other hand, in a similar way we get
\begin{eqnarray}
\mu\left\{X_N\in A\right\}&\geq& \mu\left\{X_N^J+Y_N^J\in A,\:|Y_N^J|\leq k e^{C_{15} J}\right\}\geq \mu\left\{X_N^J\in A''\right\}\geq P^J(A'')+\delta_8(N)=\nonumber\\
&=&P(A)+\delta_9(J)+\delta_8(N),\label{eq: pf main theorem 3}
\end{eqnarray}
where $A''=\{z\in A:\:|z-w|\leq k e^{-C_{15} J}, w\in A^c\}$, $\delta_8(N)\rightarrow0$ as $N\rightarrow\infty$ and $\delta_9(J)\rightarrow0$ as $J\rightarrow\infty$.
Now, taking $\lim_{N\rightarrow\infty}\lim_{J\rightarrow\infty}$, in (\ref{eq: pf main theorem 1}) and (\ref{eq: pf main theorem 3}), we obtain $\lim_{N\rightarrow\infty}\mu\left\{X_N\in A\right\}=P(A)$, i.e. (\ref{eq: main-thm-A}) as desired.
\end{proof}

\begin{remark}
Considering, as in Remark \ref{first remark}, our reference probability space $((0,1],\mathcal{B},\mu_R)$, $$\gamma_{\cdot,N}^{\bianco},\gamma_{\cdot,N}^J: \left((0,1],\mathcal{B},\mu_R\right)\rightarrow \left(\mathcal{C}([0,1],\C),\mathcal{B}_{\mathcal{C}}\right) $$
are two random function. Let $\mathrm P_N$ and $\mathrm P_N^J$ the corresponding induced probability measures on $\mathcal{C}([0,1],\C)$.
Now Proposition \ref{cor: limiting finite dim distr gammaJ}, Lemma \ref{lemma limPJ} and Theorem \ref{main-theorem} read as follows: 
for every $k\in\N$ and for every $0\leq t_1<\ldots<t_k\leq1$,

\vspace{.3cm}
\hspace{1.6cm}\xymatrix{
\mathrm P_N^J\pi_{t_1,\ldots,t_k}^{-1}\ar@2{->}[rr]^{N\rightarrow\infty}_{\tiny{\mbox{Prop. \ref{cor: limiting finite dim distr gammaJ}}}}&& \mathrm P^{(J,k)}_{t_1,\ldots,t_k}\ar@2{->}[rr]^{J\rightarrow\infty}_{\tiny{\mbox{Lem. \ref{lemma limPJ}}}}&& \mathrm P^{(k)}_{t_1,\ldots,t_k}&& \mathrm P_N\pi^{-1}_{t_1,\ldots,t_k}\ar@2{->}[ll]_{N\rightarrow\infty}^{\tiny{\mbox{Thm. \ref{main-theorem}}}}.
}
\end{remark}

\appendix
\section{Appendix: Proof of Theorem \ref{theorem:joint-limit-distr}}\label{AppendixA}
This appendix is devoted to the explanation of the proof of theorem \ref{theorem:joint-limit-distr}. This theorem is a generalization of theorem 1.6 in \cite{Cellarosi} and therefore we shall indicate how to modify its proof. Let us first recall some notation from \cite{Cellarosi}. 

Let $\hat R:\Sigma^\Z\rightarrow\Sigma^\Z$ the natural extension of $R$ as in Section \ref{section:lim-distr-K8andE} and set $D(\hat R):=\Sigma^\Z$. For $\psi\in L^1(D(\hat R))$ set $D_\Phi=\{(\hat\sigma,z):\:\hat\sigma\in D(\hat R),\,0\leq z\leq \psi(\hat \sigma)\}$, let $\{\Phi_t\}_{t\in\R}$ be the special flow on $D_\Phi$ and let $\mu_\Phi=\mu_R\times\lambda$, where $\lambda$ is the Lebesgue measure in the $z$-direction. This flow is mixing\footnote{The flow $\{\Phi_t\}_t$ is actually proven to be a K-flow.}, i.e. $\lim_{t\rightarrow\infty}\mu_\Phi(A\cap\Phi_{-t}(B))=\mu(A)\mu(B)$ for every Borel subsets $A,B\subset D_\Phi$ (see Proposition 3.4 in \cite{Cellarosi}).
We shall use the following relation between the special flow $\Phi_t$ and the (non-normalized) Birkhoff sum of $\psi$ under $\hat R$. Setting $\mathrm{S}_r^{\hat R}(\psi)(\hat \sigma):=\sum_{j=0}^{r-1}\psi(\hat R^j(\hat\sigma))$ and $r(\hat\sigma,t):=\min\{r\in\N:\:\mathrm{S}_r^{\hat R}(\psi)(\hat \sigma)>t\}$ we get for $t\in\R^+$
\begin{equation}\nonumber
\Phi_t(\hat\sigma,0)=\left(\hat R^{r(\hat \sigma,t)-1}(\hat\sigma),t-\mathrm{S}_{r(\hat\sigma,t)-1}^{\hat R}(\psi)(\hat\sigma)\right).
\end{equation}
Fix a cylinder $\mathcal{C}$ and set  $g_{\mathcal{C}}:=\sup_{\hat\sigma\in\mathcal{C}}g(\hat\sigma)$, where $g:D(\hat R)\rightarrow \R^+$ is a function defined so that
\begin{equation}\label{logqn-Birkhoff-g}\log\hat q_n(\hat\sigma)=\mathrm{S}_n^{\hat R}(\psi)(\hat \sigma)+g(\hat\sigma)+\varepsilon_n(\hat\sigma),\hspace{.5cm}\sup_{\hat\sigma\in D(\hat R)}|\varepsilon_n(\hat\sigma)|\leq C_{23}3^{-n/3}
\end{equation} for some constant $C_{23}>0$. If $|g(\hat\sigma)-g_{\mathcal{C}}|\leq \varepsilon/2$ on $\mathcal{C}$ (this is always possible, by considering a sufficiently small cylinder $\mathcal{C}$), then one can choose a time $\mathrm{T}=\mathrm{T}(N,\mathcal C)=\log N-g_{\mathcal{C}}$ so that $\hat n_N(\hat\sigma)=r(\hat\sigma,\mathrm{T})$ holds on $\mathcal{C}\smallsetminus U$, where $U=U(\mathcal{C})\subset\mathcal{C}$, $\mu_{\hat R}(U)\leq 7\varepsilon\mu_{\hat R}(C)$. Given two 
functions $F_1,F_2
:D(\hat R)\rightarrow\R$ we define 
\begin{eqnarray}
&&D_\Phi(F_1,F_2
):
=\{(\hat\sigma,z)\in D_\Phi:\:\psi(\hat\sigma)-F_2(\hat\sigma)<z<\psi(\hat\sigma)-F_1(\hat\sigma)\}.\nonumber
\end{eqnarray} Notice that for some values of $F_1(\hat\sigma),F_2(\hat\sigma)
$ (e.g. when they are negative) the corresponding sets of $z$'s can be empty.

\begin{proof}[Sketch of proof of Theorem \ref{theorem:joint-limit-distr}]
The condition $(\sigma_{\hat n_{N}+l})_{l=-N_1+1}^{N_2}=\underline c$ in (\ref{eq:thm-joint-lim}) can be rewritten as $\hat R^{\hat n_N(\hat\sigma)-1}(\hat \sigma)\in\mathcal{C}_{N_1,N_2}^{(\underline{c})}$, where $\mathcal{C}_{N_1,N_2}^{(\underline{c})}$ is a cylinder determined by $N_1,N_2$ and $\underline{c}$. 
We claim that 
\begin{eqnarray}
&&\hspace{-.68cm}\lim_{N\rightarrow\infty}\!\mu_R\!\left(\!\left\{\!\alpha\!\in\!(0,1]\!:a_1\!<\!\frac{\hat q_{\hat n_N-1}}{N}\!<\!b_1,\,a_2\!<\!\frac{\hat q_{\hat n_N}}{N}\!<\!b_2,\,\hat R^{\hat n_N(\hat\sigma)-1}(\hat \sigma)\!\in\!\mathcal{C}_{N_1,N_2}^{(\underline{c})},\scriptsize{\left(\!\!\begin{array}{l}x_{\hat n_N-N_1}\\y_{\hat n_N-N_1}\end{array}\!\!\!\right)\!=\!\left(\!\!\begin{array}{l}x\\y\end{array}\!\!\right)}\!\right\}\!\right)\!=\nonumber\\
&&\hspace{-.68cm}=p_{x,y,\underline c}\cdot\mu_{\Phi}\left(\bar D_{\Phi}(a_1,b_1,a_2,b_2,\underline c)\right),\label{proof claim}
\end{eqnarray}
where $p_{x,y,\underline c}$ is a real number between 0 and 1 (we shall define it later in this proof), $\bar D_\Phi(a_1,b_1,a_2,b_2):=D_\Phi(\log a_1+\psi\circ\hat R^{-1},\log b_1+\psi\circ\hat R^{-1})\cap\,D_\Phi(\log a_2,\log b_2)\cap\,p^{-1}\mathcal{C}_{N_1,N_2}^{(\underline c)}$ (see Figure \ref{fig: regionD}) and $p:D_\Phi\rightarrow D(\hat R)$ is the vertical projection onto the base. 
\begin{figure}
\begin{center}
\vspace{-6cm}
\includegraphics[width=15.5cm, angle=0]{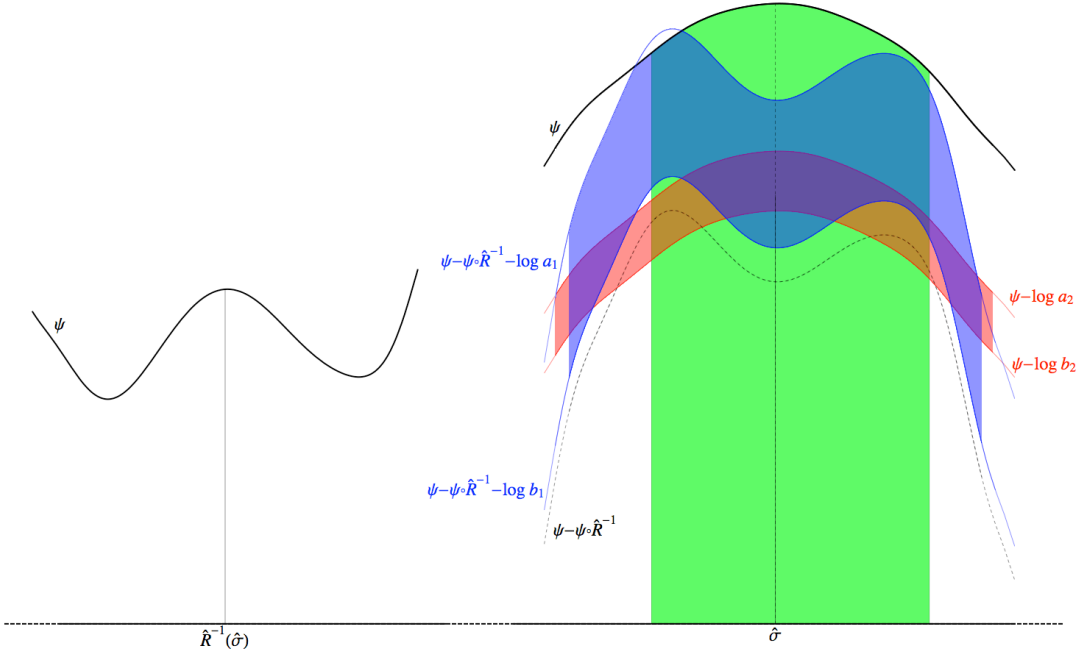}
\vspace{-6.0cm}
\caption{\small{The region $\bar D_\Phi(a_1,b_1,a_2,b_2,\underline c)$ described in the proof of Theorem \ref{theorem:joint-limit-distr} is the intersection of the three shaded regions: $D_\Phi(\log a_1+\psi\circ\hat R^{-1},\log b_1+\psi\circ\hat R^{-1})$ (blue), $D_\Phi(\log a_2,\log b_2)$ (red) and $p^{-1}\mathcal{C}_{N_1,N_2}^{(\underline c)}$ (green).}}\label{fig: regionD}
\end{center}
\end{figure}
Set $$A_{\mathcal{C}}:=\left\{\hat\sigma\in\mathcal{C}:a_1\!<\!\frac{\hat q_{\hat n_N-1}}{N}\!<\!b_1,\,a_2\!<\!\frac{\hat q_{\hat n_N}}{N}\!<\!b_2,\hat R^{\hat n_N(\hat\sigma)-1}(\hat \sigma)\in\mathcal{C}_{N_1,N_2}^{(\underline c)},\:\scriptsize{\left(\!\!\begin{array}{l}x_{\hat n_N-N_1}\\y_{\hat n_N-N_1}\end{array}\!\!\right)\!=\!\left(\!\!\begin{array}{l}x\\y\end{array}\!\!\right)}\right\}.$$ 
Consider $\varepsilon>0$. One can find a finite collection of cylinders $\frak{C}_\varepsilon$ for which  (\ref{eq:thm-joint-lim}) can be $10\varepsilon$-approximated by$\sum_{\mathcal{C}\in\frak{C}_\varepsilon}\mu_{\hat R}(A_{\mathcal{C}\smallsetminus U})$,
%
where $U=U(\mathcal{C})$ is as above. 
In order to show (\ref{proof claim}), noticing that $A_{\mathcal{C}}$ depends on $N$, it is enough to prove that, for sufficiently large $N$, 
\begin{eqnarray}
\left|\frac{\mu_{\hat R}\left(A_{\mathcal C\smallsetminus U}\right)}{\mu_{\hat R}\left(\mathcal C\smallsetminus U\right)}-p_{x,y,\underline c}\cdot\mu_{\Phi}\left(\bar D_{\Phi}(a_1,b_1,a_2,b_2,\underline c)\right)
\right|\leq C_{24}\varepsilon\label{proof fraction}
\end{eqnarray}
for some $C_{24}>0$. If $N$ is sufficiently large we get
\begin{eqnarray}
&&\left\{\hat\sigma\in\mathcal{C}\smallsetminus U:\:a_1<\frac{\hat q_{\hat n_N-1}}{N}<b_1,\:a_2<\frac{\hat q_{\hat n_N}}{N}<b_2\right\}=\nonumber\\
&&=\left\{\hat\sigma\in\mathcal{C}\smallsetminus U:\:\log a_1<\mathrm{S}_{r(\hat\sigma,\mathrm{T})-1}^{\hat R}(\psi)(\hat\sigma)-\mathrm{T}+\varepsilon_{N,\mathcal{C}}(\hat\sigma)<\log b_1,\right\}\cap\nonumber\\
&&\hspace{.5cm}\cap\left\{\hat\sigma\in\mathcal{C}\smallsetminus U:\:\log a_2<\mathrm{S}_{r(\hat\sigma,\mathrm{T})}^{\hat R}(\psi)(\hat\sigma)-\mathrm{T}+\varepsilon'_{N,\mathcal{C}}(\hat\sigma)<\log b_2,\right\},\nonumber
\end{eqnarray}
where $\varepsilon_{N,\mathcal{C}}(\hat\sigma):=\varepsilon_{\hat n_N(\hat\sigma)-1}(\hat\sigma)-g_{\mathcal{C}}+g(\hat\omega)$, $\varepsilon'_{N,\mathcal{C}}(\hat\sigma):=\varepsilon_{\hat n_N(\hat\sigma)}(\hat\sigma)-g_{\mathcal{C}}+g(\hat\omega)$ and $\varepsilon_{\hat n_N(\hat\sigma)-1}$, $\varepsilon_{\hat n_N(\hat\sigma)}$ are defined in (\ref{logqn-Birkhoff-g}). One can show that $\sup_{\hat\sigma\in\mathcal{C}\smallsetminus U}|\varepsilon_{N,\mathcal{C}}(\hat\sigma)|+\sup_{\hat\sigma\in\mathcal{C}\smallsetminus U}|\varepsilon'_{N,\mathcal{C}}(\hat\sigma)|\leq C_{25}\varepsilon$ for some $C_{25}>0$. Notice that $v:=\mathrm{S}_{r(\hat\sigma,\mathrm{T})}^{\hat R}(\psi)(\hat\sigma)-\mathrm{T}$ is the vertical distance from $\Phi_T(\hat\sigma,0)$ and the roof function $\psi(\hat R^{\hat n_N(\hat\sigma)-1}(\hat\sigma))$ and therefore $\mathrm{S}_{r(\hat\sigma,\mathrm{T})-1}^{\hat R}(\psi)(\hat\sigma)-\mathrm{T}=v-\psi(\hat R^{\hat n_N(\hat\sigma)-2}(\hat\sigma))$. Using the vertical projection $p:D_\Phi\rightarrow D(\hat R)$ we write the condition $\hat R^{\hat n_N(\hat\sigma)-1}(\hat \sigma)\in\mathcal{C}_{N_1,N_2}^{(\underline c)}$ as $p(\Phi_T(\hat\sigma,0))\in\mathcal{C}_{N_1,N_2}^{(\underline c)}$ and setting $B_N(x,y)
:=\{\hat\sigma\in D(\hat R):\:x_{\hat n_N(\hat\sigma)-N_1}(\hat\sigma)=x,\,y_{\hat n_N(\hat\sigma)-N_1}(\hat\sigma)=y\}$
 we get
\begin{eqnarray}
A_{\mathcal{C}\smallsetminus U}\times\{0\}&\subseteq&\left((\mathcal{C}\smallsetminus U)\times\{0\}\right)\cap\left(B_N(x,y)\times\{0\}\right)\cap\nonumber\\
&&
\cap\,\Phi_{-T}\big(D_\Phi(\log a_1+\psi\circ\hat R^{-1}-C_{25}\varepsilon,\log b_1+\psi\circ\hat R^{-1}+C_{25}\varepsilon)\cap \nonumber\\
&&\hspace{1.5cm}\cap\,D_\Phi(\log a_2-C_{25}\varepsilon,\log b_2+C_{25})\cap p^{-1}\mathcal{C}_{N_1,N_2}^{(\underline c)}\big)\nonumber
\end{eqnarray}
and
\begin{eqnarray}
A_{\mathcal{C}\smallsetminus U}\times\{0\}&\supseteq&\left((\mathcal{C}\smallsetminus U)\times\{0\}\right)\cap\left(B_N(x,y)\times\{0\}\right)\cap\nonumber\\
&&
\cap\,\Phi_{-T}\big(D_\Phi(\log a_1+\psi\circ\hat R^{-1}+C_{25}\varepsilon,\log b_1+\psi\circ\hat R^{-1}-C_{25}\varepsilon)\cap \nonumber\\
&&\hspace{1.5cm}\cap\,D_\Phi(\log a_2+C_{25}\varepsilon,\log b_2-C_{25})\cap p^{-1}\mathcal{C}_{N_1,N_2}^{(\underline c)}\big).\nonumber
\end{eqnarray}
For sufficiently small $\delta$, $0<\delta<\varepsilon$, one can show that 
\begin{eqnarray}
A_{\mathcal{C}\smallsetminus U}\times[0,\delta)&\subseteq&\Phi_{-T}\big((D_{\Phi}(\log a_1+\psi\circ\hat R^{-1}-C_{25}\varepsilon-\delta,\log b_1+\psi\circ\hat R^{-1}+C_{25}\varepsilon)\cap\nonumber\\ 
&&\hspace{1.2cm}\cap\, D_\Phi(\log a_2-C_{25}\varepsilon-\delta,\log b_2+C_{25}\varepsilon)\cap p^{-1}\mathcal{C}_{N_1,N_2}^{(\underline c)})\cup D_\Phi^\delta\big),\nonumber
\end{eqnarray}
where $D_\Phi^\delta:=D(\hat R)\times[0,\delta)$. Thus, recalling that $\mathrm T=\mathrm T(N)=\log N-g_{\mathcal{C}}$ 
and setting $W_N^+(\varepsilon,\delta):=\Phi_{-T}\big(\bar D^{\varepsilon,+}_\Phi(a_1,b_1,a_2,b_2,\underline{c})\cup D_\Phi^\delta\big)$, where $\bar D^{\varepsilon,+}_\Phi(a_1,b_1,a_2,b_2,\underline{c}):=(D_{\Phi}(\log a_1+\psi\circ\hat R^{-1}-C_{26}\varepsilon,\log b_1+\psi\circ\hat R^{-1}+C_{25}\varepsilon)\cap D_\Phi(\log a_2-C_{26}\varepsilon,\log b_2+C_{25}\varepsilon)\cap p^{-1}\mathcal{C}_{N_1,N_2}^{{(\underline c)}})$ and $C_{26}=C_{25}+1$, we obtain
\begin{eqnarray}
&&\delta\cdot\mu_{\hat R}(A_{\mathcal{C}\smallsetminus U})\leq\mu_\Phi\big(\left((\mathcal{C}\smallsetminus U)\times[0,\delta)\right)\cap(B_N(x,y)\times[0,\delta))\cap W_N^+(\varepsilon,\delta)\big)\label{eq: before-mixing1}.
\end{eqnarray}
Our goal is to show that, for sufficiently large $N$, one can $C_{27}\varepsilon$-approximate (for some constant $C_{27}>0$) the left hand side of (\ref{eq: before-mixing1}) with the product of the $\mu_\Phi$-measures of the three sets $(\mathcal{C}\smallsetminus U)\times[0,\delta)$, $B_N(x,y)\times[0,\delta)$ and $W_N^+(\varepsilon,\delta)$. 
First, we can replace $B_N(x,y)\times[0,\delta)$ in  (\ref{eq: before-mixing1}) by $B_N'(x,y):=B_N(x,y)\times\{(\hat\sigma,z)\in D_\Phi:0\leq z\leq\psi(\hat\sigma)\}$ and write $D_N=D_N(x,y,a_1,b_1,a_2,b_2,\underline{c},\varepsilon,\delta):=B_N'\cap W^+_N(\varepsilon,\delta)=\Phi_{-\mathrm T(N)}(E_N)$, where $E_N:=\Phi_{\mathrm T(N)}(B_N')\cap \bar D^{\varepsilon,+}_\Phi(a_1,b_1,a_2,b_2,\underline{c})$.

Let us recall the following classical result by R\'enyi \cite{Renyi1958}: let $(\Omega,\frak B, P)$ be a probability space and let $G,H_N\in\frak{B}$, $N\in\N$, then
\begin{equation}\label{Renyi}
\lim_{N\rightarrow\infty}P(G\cap H_N)\rightarrow P(A)\cdot\beta\hspace{.5cm}\mbox{iff}\hspace{.5cm}\lim_{N\rightarrow\infty}P(H_k\cap H_N)=P(H_k)\cdot\beta\hspace{.5cm}\mbox{for each $k\in\N_0$},
\end{equation} 
where $H_0=\Omega$. In our case $\Omega=D_\Phi$, $P=\mu_\Phi$, $A=(\mathcal{C}\smallsetminus U)\times[0,\delta)$ and $H_N=D_N$. We can compute $P(H_k\cap H_N)$ for fixed $k$ as follows.
\begin{eqnarray}
\mu_\Phi(D_k\cap D_N)=\mu_\Phi\big(\Phi_{-\mathrm{T}(k)}\left(E_k\cap\Phi_{-(\mathrm T(N)-\mathrm{T}(k))}(E_N)\right)\big)=\mu_\Phi\big(E_k\cap\Phi_{-(\mathrm T(N)-\mathrm{T}(k))}(E_N)\big)\label{eq: before-mixing2}
\end{eqnarray}
For every $k\in\N$ we can write $E_k$ as a disjoint union of 
\begin{eqnarray}
E_k^{(\overline n,\underline\theta)}&\!\!:=\!\!&\{(\hat\sigma,y)\in D_\Phi: \: \hat\sigma=\hat R^{\hat n_k(\hat\sigma')-N_1}(\hat\sigma'),\:\hat n_k(\hat\sigma')=\overline n,\:(\hat\sigma'_j)_{j=1}^{\overline n-N_1}=\underline\theta\}\cap\nonumber\\
&&\cap\, \bar D^{\varepsilon,+}_\Phi(a_1,b_1,a_2,b_2,\underline{c}),\nonumber
\end{eqnarray}
where $\overline n\in\N$ and $\underline\theta\in \Sigma^{\overline n-N_1}$ is such that $x_{\overline n-N_1}(\underline\theta)=x$ and $y_{\overline n-N_1}(\underline\theta)=y$ and we can write (\ref{eq: before-mixing2}) as
\begin{eqnarray}
&&\mu_\Phi\big(E_k\cap\Phi_{-(\mathrm T(N)-\mathrm{T}(k))}(E_N)\big)=\sum_{\overline n,\underline\theta}\mu_\Phi\!\left(E_k^{(\overline n,\underline\theta)}\cap\Phi_{-(\mathrm T(N)-\mathrm{T}(k))}(E_N)\right).\label{eq: before-mixing2b}
\end{eqnarray}
Each term in the series above is now written as a product
\begin{eqnarray}\mu_\Phi\!\left(\Phi_{\mathrm{T}(k)}(B_N')\big|E_k^{(\overline n,\underline\theta)}\cap \Phi_{-(\mathrm T(N)-\mathrm{T}(k))}(\bar D^{\varepsilon,+}_\Phi(a_1,b_1,a_2,b_2,\underline{c}))\right)\cdot\label{eq: before-mixing3}\\
\cdot\mu_\Phi\!\left(E_k^{(\overline n,\underline\theta)}\cap \Phi_{-(\mathrm T(N)-\mathrm{T}(k))}(\bar D^{\varepsilon,+}_\Phi(a_1,b_1,a_2,b_2,\underline{c}))\right).\label{eq: before-mixing4}
\end{eqnarray}
We apply the the mixing property of the special flow $\{\Phi_t\}$ to the factor (\ref{eq: before-mixing4}), getting
\begin{eqnarray}
\mu_\Phi\!\left(E_k^{(\overline n,\underline\theta)}\cap \Phi_{-(\mathrm T(N)-\mathrm{T}(k))}(\bar D^{\varepsilon,+}_\Phi(a_1,b_1,a_2,b_2,\underline{c}))\right)\rightarrow\mu_\Phi\!\left(E_k^{(\overline n,\underline\theta)}\right)\mu_\Phi\!\left(\bar D^{\varepsilon,+}_\Phi(a_1,b_1,a_2,b_2,\underline{c})\right).\nonumber
\end{eqnarray}
as $N\rightarrow\infty$.
We claim that the factor (\ref{eq: before-mixing3}) also has a limit:
\begin{eqnarray}
\lim_{N\rightarrow\infty}\mu_\Phi\!\left(\Phi_{\mathrm{T}(k)}(B_N'(x,y))\big|E_k^{(\overline n,\underline\theta)}\cap \Phi_{-(\mathrm T(N)-\mathrm{T}(k))}(\bar D^{\varepsilon,+}_\Phi(a_1,b_1,a_2,b_2,\underline{c}))\right)=:p_{x,y,\underline{c}}.\label{limit pxyc}
\end{eqnarray}
In order to see this one can analyze geometrically the action of the special flow as follows.
The set $E_k^{(\overline n,\underline\theta)}$ is fixed and involves a finite number of entries of $\hat\sigma^-$ in the base $D(\hat R)$ and some region in the $z$-direction. In the $D(\hat R)$ component, the set $ \Phi_{-(\mathrm T(N)-\mathrm{T}(k))}(D^{\varepsilon,+}_\Phi(a_1,b_1,a_2,b_2,\underline{c}))$ corresponds to setting to $\underline{c}$ the coordinates at from $(\hat\sigma_j)_{j=\hat n_N-\overline n-N_1+1}^{\hat n_N-\overline n+N_2}$, i.e. in a neighborhood (of fixed size) of the renewal time $\hat n_N$. In the $z$-direction it gives a region which, by mixing, spreads according to the invariant measure $\mu_\Phi$ as $N\rightarrow\infty$. Since the set $\Phi_{\mathrm T(k)}(B_N'(x,y))$ gives no restrictions in the $z$-direction, it is enough to establish the existence of the limit (\ref{limit pxyc}) for the projection of the sets onto the base $D(\hat R)$. In the base, however, the limit follows from the Markov-like property of the process $\{(x_n,y_n)\}_{n\in\N}\in\Xi^{\N}$ (namely extending (\ref{P0(xn,yn) as matrix power}) to conditional probability distributions
). Now taking the limit in 
(\ref{eq: before-mixing2b}) 
we get
\begin{eqnarray}
\lim_{N\rightarrow\infty}\mu_\Phi\left(E_k\cap\,\Phi_{-(\mathrm T(N)-\mathrm T(k))}(E_N)\right)&=&p_{x,y,\underline{c}}\cdot\mu_\Phi\!\left(\bar D^{\varepsilon,+}_\Phi(a_1,b_1,a_2,b_2,\underline{c})\right)\sum_{\overline n,\underline\theta}\mu_\Phi\!\left(E_k^{(\overline n,\underline\theta)}\right)=\nonumber\\
&=&p_{x,y,\underline{c}}\cdot\mu_\Phi\!\left(\bar D^{\varepsilon,+}_\Phi(a_1,b_1,a_2,b_2,\underline{c})\right)\cdot\mu_\Phi(E_k),\nonumber
\end{eqnarray}
i.e. the rightmost part of (\ref{Renyi}) with $\beta=p_{x,y,\underline{c}}\cdot\mu_\Phi\!\left(\bar D^{\varepsilon,+}_\Phi(a_1,b_1,a_2,b_2,\underline{c})\right)$. Thus we proved that
\begin{eqnarray}
&&\lim_{N\rightarrow\infty}\mu_\Phi\big(\left((\mathcal{C}\smallsetminus U)\times[0,\delta)\right)\cap(B_N(x,y)\times[0,\delta))\cap W_N^+(\varepsilon,\delta)\big)=\nonumber\\
&&=\mu_\Phi\big((\mathcal C\smallsetminus U)\times[0,\delta)\big)\cdot p_{x,y,\underline{c}}\cdot\mu_\Phi\!\left(\bar D^{\varepsilon,+}_\Phi(a_1,b_1,a_2,b_2,\underline{c})\right)=\nonumber\\
&&=\delta\cdot\mu_{\hat R}(\mathcal C\smallsetminus U)\cdot p_{x,y,\underline{c}}\cdot\mu_\Phi\!\left(\bar D^{\varepsilon,+}_\Phi(a_1,b_1,a_2,b_2,\underline{c})\right).\label{eq: after-mixing1}
\end{eqnarray}
Now (\ref{eq: before-mixing1}) and (\ref{eq: after-mixing1}) imply that, for sufficiently large $N$,
\begin{eqnarray}
\delta\cdot\mu_{\hat R}(A_{\mathcal C\smallsetminus U})\leq\delta\cdot\mu_{\hat R}\left(\mathcal C\smallsetminus U\right)\cdot\left(p_{x,y,\underline{c}}\cdot\mu_\Phi\!\left(\bar D^{\varepsilon,+}_\Phi(a_1,b_1,a_2,b_2,\underline{c})\right)+C_{27}\varepsilon\right),\label{proof inequality1}
\end{eqnarray}
for some $C_{27}>0$. Proceeding as in \cite{Cellarosi} (Lemma 3.8 therein) one can show that, for sufficiently small $\delta$,
\begin{eqnarray}
&&\left(\mathcal C\smallsetminus U)\times[0,\delta)\right)\cap
\,\Phi_{-\mathrm T}
\big(
\bar D_{\Phi}^{\varepsilon,-}(a_1,b_1,a_2,b_2,\underline c)\smallsetminus D_\Phi^{\delta}
\big)\subseteq
A_{\mathcal C\smallsetminus U}\times[0,\delta),\nonumber
\end{eqnarray}
where $\bar D_{\Phi}^{\varepsilon,-}=D_\Phi(\log a_1+\psi\circ\hat R^{-1}+C_{28}\varepsilon,\log b_1+\psi\circ\hat R^{-1}-C_{29}\varepsilon)
\cap D_\Phi(\log a_2+C_{28}\varepsilon,\log b_2-C_{29}\varepsilon)\cap p^{-1}\mathcal C_{N_1,N_2}^{(\underline c)}$, for some $C_{28},C_{29}>0$. Using the mixing property of the flow $\{\Phi_t\}_t$ as above we get, for sufficiently large $N$,
\begin{eqnarray}
\delta\cdot\mu_{\hat R}(A_{\mathcal C\smallsetminus U})\geq\delta\cdot\mu_{\hat R}(\mathcal C\smallsetminus U)\cdot\left(p_{x,y,\underline{c}}\cdot\mu_\Phi\left(\bar D_{\Phi}^{\varepsilon,-}(a_1,b_1,a_2,b_2,\underline c)\right)-C_{30}\varepsilon\right)\label{proof inequality2}
\end{eqnarray}
for some $C_{30}>0$. Moreover, by Fubini's Theorem, for some $C_{31}>0$,
\begin{eqnarray}
\left|\mu_\Phi\!\left(\bar D_{\Phi}^{\varepsilon,\pm}(a_1,b_1,a_2,b_2,\underline c)\right)-p_{x,y,\underline c}\cdot\mu_\Phi\!\left(\bar D_\Phi(a_1,b_1,a_2,b_2,\underline c)\right)\right|\leq C_{31}\varepsilon.\label{Fubini}
\end{eqnarray}
Finally, by (\ref{proof inequality1},\ref{proof inequality2},\ref{Fubini}) we get (\ref{proof fraction}) for some $C_{24}>0$ and this completes the proof of Theorem \ref{theorem:joint-limit-distr}.
\end{proof}

\section*{Acknowledgments}
Part of this work was completed while at the Erwin Schr\"{o}dinger Institute for Mathematical Physics, Vienna, and at the Abdus Salam International Center for Theoretical Physics, Trieste.

I am very grateful to Fr\'{e}d\'{e}ric Klopp, who first presented me the topic of theta sums and curlicues.
Moreover, I want to thank Alexander Bufetov, Giovanni Forni, Stefano Galatolo, Marco Lenci and 
Ilya Vinogradov 
for many useful discussions.
Lastly, I wish to express my sincere gratitude to my advisor Yakov G. Sinai for his invaluable support and guidance.

\addcontentsline{toc}{chapter}{Bibliography}
\bibliographystyle{plain}
\bibliography{bibliography}
\end{document}